\documentclass[12pt]{article}
\usepackage{amsmath, amssymb,amsxtra}
\usepackage{enumerate}
\usepackage{amsthm}
\usepackage{stackrel}
\usepackage{tabularx}% http://ctan.org/pkg/tabularx
\theoremstyle{plain}

\numberwithin{equation}{section}

\newtheorem{thm}{{\sc Theorem}}[section]

\newtheorem{cor}[thm]{{\sc Corollary}}
\newtheorem{lem}[thm]{{\sc Lemma}}
\newtheorem{df}[thm]{{\sc Definition}}

\newcommand{\comm}[1]{}

\newcommand{\n}{\hbox{${\cal N}$}}
\newcommand{\K}{\hbox{${\cal K}$}}
\newcommand{\Gr}{\hbox{${\cal G}r$}}
\newcommand{\oi}{\hbox{${\mathcal O}$}}
\newcommand{\la}{\hbox{$\Lambda$}}
\newcommand{\m}{\hbox{${\cal M}$}}
\newcommand{\LR}{\hbox{Little\-wood-Richard\-son}}

\begin{document}
\title{An Elementary Construction of a Hive Associated to a Hermitian Matrix Pair, with Applications}
\author{Glenn D.\ Appleby, Tamsen Whitehead\\
Department of Mathematics\\
Santa Clara University\\
Santa Clara, CA 95053\\
{\tt gappleby@scu.edu, tmcginley@scu.edu}}
\date{ }
%\maketitle
\newpage

\noindent {\bf Hives Determined by Pairs in the Affine Grassmannian over Discrete Valuation Rings}

\medskip

\noindent
Glenn D. Appleby\\
 Tamsen Whitehead\\
{\em Department of Mathematics\\
and Computer Science,\\
Santa Clara University\\
Santa Clara,  CA 95053}\\
gappleby@scu.edu, tmcginley@scu.edu

\noindent \mbox{} \hrulefill \mbox{}
\begin{abstract}
Let ${\mathcal O}$ be a discrete valuation ring with quotient field ${\cal K}$. The \emph{affine Grassmannian} ${\cal G}r$ is the set of full-rank ${\mathcal O}$-modules contained in ${\cal K}^n$.
Given $\Lambda \in {\cal G}r$, invariant factors $inv(\Lambda)=\lambda \in {\mathbb Z}^n$ stratify ${\cal G}r$. Left-multiplication by $GL_{n}({\cal K})$  stratifies ${\cal G}r \times {\cal G}r$ where $inv(N,\Lambda) = \mu$ if $(N,\Lambda)$ and $(I_{n} ,M)$ are in the same $GL_{n}({\cal K})$ orbit, and $inv(M) = \mu$. We present an elementary map from ${\cal G}r \times {\cal G}r$ to hives (in the sense of Knutson and Tao) of type $(\mu,\nu,\lambda)$ where $inv(N,\Lambda) = \mu$, $inv(N) = \nu$, and $inv(\Lambda) = \lambda$. Earlier work by the authors \cite{Ours-matrix to filling} determined Littlewood-Richardson fillings from matrix pairs over certain rings $\oi$, and later Kamnitzer \cite{kam} utilized properties of MV polytopes to define a map from $\Gr \times \Gr$ to hives over ${\mathcal O} = {\mathbb C}[[t]]$. Our proof uses only linear algebra methods over any discrete valuation ring, where hive entries are minima of sums of orders of invariant factors over certain submodules. Our map is analogous to a conjectured construction of hives from Hermitian matrix pairs due to Danilov and Koshevoy \cite{DK}.

\end{abstract}

\noindent \mbox{} \hrulefill \mbox{}
\section{Introduction} Let $\oi$ be a discrete valuation ring with quotient field $\K$. The \emph{affine Grassmannian} over $\oi$, denoted $\Gr$, is the set of full-rank $\oi$ modules contained in $\K^n$ (sometimes called \emph{lattices} in $\K^n$). The left quotient $GL_{n}(\oi) \backslash GL_{n}(\K)$ may be identified with $\Gr$ by associating to each coset in $GL_{n}(\oi) \backslash GL_{n}(\K)$ the common $\oi$-module spanned by the columns of any element in the coset. For elements $\Lambda \in \Gr$, the invariant factors $inv(\Lambda) = \lambda =(\lambda_{1}, \lambda_2, \ldots , \lambda_n)$ where $\lambda_{i} \in {\mathbb Z}$ and $\lambda_i \geq \lambda_{i+1}$ allow us to stratify $\Gr$. Letting $Gl_{n}( \K)$ act by right-multiplication on pairs $(N,\Lambda) \in \Gr \times \Gr$ also allows us to stratify orbits similarly by setting $inv(\n,\Lambda) = \mu$ whenever $(\n,\Lambda)$ and $(I_{n} ,\m)$ are in the same $GL_{n}(\K)$ orbit, and $inv(\m) = \mu$.

This paper determines a hive, in the sense of \cite{knut}, from a given pair $(\n,\la) \in \Gr \times \Gr$ over an arbitrary discrete valuation ring $\oi$. Hives are triangular arrays of integers $\{ h_{st} \}$ satisfying certain linear inequalities. These discrete objects are used to count or classify objects from other areas of mathematics, particularly in representation theory and algebraic combinatorics. There exist simple linear bijections (see~\cite{pak}) given by integer-valued matrices that transform hives to the well-known class of \LR\ fillings of skew tableaux, which often serve the same purpose. It becomes a matter of convenience which class of objects to use, and the determination is made on the basis of how naturally the combinatorial object may be identified with objects in some other class of problems.

Let us fix a uniformizing parameter $t \in \oi$, and let $n \times n$ matrices $M, N \in GL_{n}(\oi)$. The orders (with respect to $t$) of the invariant factors of $M$ and $N$ may be denoted by non-increasing sequences of \emph{non-negative} integers, which we shall call the \emph{invariant partitions} uniquely determined by these matrices. Denote the invariant partition of $M$ by $inv(M)=\mu$, that of $N$ by $inv(N) = \nu$, and that of the product $MN$ by $inv(MN)=\lambda$. Well-known results~\cite{klein,mac} show that the triple of invariant partitions $(\mu, \nu, \lambda)$ may be realized by matrices $M$, $N$, and $MN$ if and only if the \emph{Littlewood-Richardson coefficient} $c_{\mu \nu}^{\lambda} \neq 0$. This coefficient may also be combinatorially determined by the \emph{Littlewood-Richardson Rule}, which states that $c_{\mu \nu}^{\lambda}$ equals the number \emph{Littlewood-Richardson fillings} of a skew-shape $\lambda / \mu$ with content $\nu$ (we say the filling is of \emph{type} $(\mu,\nu, \lambda)$).

In ~\cite{Ours-matrix to filling}, the authors proved a determinantal formula to compute a Littlewood-Richardson filling from a pair $(M,N)$ in $GL_{n}(\oi)$, under some hypotheses regarding the discrete valuation ring $\oi$. The first author's work \cite{me} constructed matrix realizations $(M,N)$ from an arbitrary \LR\ filling of type $(\mu, \nu, \lambda)$. In the authors' work \cite{lrreal} these results were extended to matrix pairs $(\m,\n)$ in $GL_{n}(\K)$ by defining \LR\ fillings admitting some negative, real-valued entries, also showing that this matrix setting realized combinatorial bijections establishing $c_{\mu \nu}^{\lambda} = c_{\nu \mu}^{\lambda}$. In Kamnitzer \cite{kam} these issues were formulated in terms of elements of the affine Grassmannian $\Gr$ over $\oi = {\mathbb C}[[t]]$, where from each pair $(\n , \la) \in \Gr$ a hive (defined below) was determined, among other results.

Our method for determining a hive from pairs $(\n , \la) \in \Gr \times \Gr$ is to compute the maxima of the orders (sums of invariant factors) of certain submodules of $\n$ and $\la$, subject to specified constraints on their ranks (an alternate formula is also obtained using \emph{minima} of a different collection of submodules). The formula for these maxima is quite similar to a conjectured formula put forward by Danilov and Koshevoy in a different (but related) setting.
 Danilov and Koshevoy~\cite{DK} conjectured that if $M$ and $\Lambda$ were two $n\times n$ Hermitian matrices (over ${\mathbb C}$), one may obtain a hive $\{ h_{st} \}$, for $0 \leq s \leq t \leq n$ by setting
\[ h_{st} = \max_{U_{s}\oplus V_{t-s}} tr(\Lambda|_{U_{s}}) + tr(M|_{V_{t-s}}), \]
 where $U_{s}$ and $V_{t-s}$ are orthogonal subspaces of ranks $s$ and $t-s$, respectively, and $tr(M|_{U_{s}})$ denotes the trace of $M$ restricted to the subspace $U_{s}$, etc.

 Our formula replaces the trace of a Hermitian matrix restricted to an $s$-dimensional subspace with the sums of orders of invariant factors of a rank-$s$ submodule. A similar construction (in the context of tropical geometry) was studied by Speyer~\cite{speyer}. In the authors' abstract~\cite{herm2hive} we described efforts to use the formula of Danilov and Koshevoy in the context of continuous deformations of eigenvalues of Hermitian matrix pairs (by rotations of axes of eigenvectors) to generate (by means of the formula) paths in an associated ${\mathfrak sl}_{n}$ crystal. However, while the examples of connections between Hermitian spectral deformation and crystals remain intriguing, our purported proof of the hive formula in the Hermitian case as stated in the abstract was in error.

In this paper, our proof of the hive formula for affine Grassmannians is based on an analysis of the determinantal formulas appearing in the authors' earlier work~\cite{Ours-matrix to filling}. However, our formula below, and its proof, are substantially simpler. Furthermore, we are able to relax the hypotheses on $\oi$ appearing in that earlier work, so that our hive construction is defined over an arbitrary discrete valuation ring.  In fact, the proof presented here could be applied with little change to a valuation ring with ${\mathbb R}$ as the valuation group.

\medskip
Let us state our main result.
Given some $\oi$-submodule $U \subseteq K^n$, we will let $\| U \|$ denote the sums of the orders of the invariant factors of $U$ (precise definitions given below). With this, we shall prove:

\begin{thm}
\label{main theorem}
  Let $\Lambda, \n \in \Gr$ be two full $\oi$-lattices of $K^n$. Let the invariant partition of $\n$ be $inv(\n) = \nu = (\nu_{1} , \ldots , \nu_{n})$ and $inv(\Lambda) = \lambda = (\lambda_{1} , \ldots , \lambda_{n})$. Suppose $(I,{\cal M})$ and $(\n,\Lambda)$ are in the same $GL_{n}(\K)$ orbit, where $inv(\m) = \mu =  (\mu_{1}, \ldots , \mu_{n} )$. Let $| \lambda | = \lambda_{1} + \cdots + \lambda_{n}$. Below, let $\la_{s}$ denote an $\oi$-submodule of $\la$ of rank $s$, and $\n_{t}$ denote an $\oi$-submodule of rank $t$ of $\n$, etc.

   Then the numbers $\{ h_{st} \}$, where
   \begin{align}   h_{st} &=  \  | \lambda | - \min_{ \Lambda_{n-t} \oplus {\scriptstyle{\cal N}_{t-s}}} \!\!\big(  \left\| \Lambda_{n-t} \oplus \n_{t-s} \right\| \big) \label{mu,nu,lambda min} \\
   & = \qquad \max_{\Lambda_{s} \oplus  {\cal M}_{t-s}} \!\!\left( \left\|  \Lambda_{s} \oplus  {\cal M}_{t-s} \right\| \right)\label{mu,nu,lambda max} %=&\  |\lambda| - \left\| \Lambda\!\!\downarrow_{n}^{t+1} + N\!\downarrow_{n}^{n-t+s+1} \right\|
     \end{align}
     form a hive of type $(\mu,\nu,\lambda)$, and the numbers $\{ h_{st}' \}$, where

\begin{align}   h_{st}' &=  \  | \lambda | - \min_{\substack{ \Lambda_{n-t} \oplus {\scriptstyle{\cal M}_{t-s}}}} \!\!\big(  \left\| \Lambda_{n-t}\oplus {\cal M}_{t-s} \right\| \big) \label{nu,mu,lambda min}\\
   & = \qquad \max_{\Lambda_{s} \oplus  {\scriptstyle{\cal N}_{t-s}}} \!\!\left( \left\|  \Lambda_{s} \oplus  \n_{t-s} \right\| \right) \label{nu,mu,lambda max}%=&\  |\lambda| - \left\| \Lambda\!\!\downarrow_{n}^{t+1} + N\!\downarrow_{n}^{n-t+s+1} \right\|
     \end{align}
  form a hive of type $(\nu,\mu,\lambda)$.
\label{full theorem}
\end{thm}

\section{Notation and Preliminary Definitions}

Let $\oi$ be a fixed discrete valuation ring, with multiplicative group of units $\oi^{\times} \subseteq \oi$, and with a fixed uniformizing parameter $t \in \oi$, so that every element $\alpha \in \oi$ may be written $\alpha = u t^k$ for some unit $u \in \oi^{\times}$ and some non-negative integer $k$. We will let $K = \oi[t^{-1}]$ denote the quotient field of the domain $\oi$. Similary, each element in $\beta \in K$ may be expressed $\beta = v t^{k'}$ for some unit $v \in \oi^{\times}$ and some $k' \in {\mathbb Z}$.

\bigskip
\noindent \underline{Notational Convention:} \emph{Since a clear distinction among the various ranks of the submodules employed in Theorem~\ref{main theorem} is essential, we adopt the convention that whenever an $\oi$-module of $K^n$ is written with a \emph{subscript}, the subscript will always denote the \emph{rank} of the submodule, so $U_{k}$ will always mean a submodule of rank $k$. Other means of distinguishing submodules from each other will use different letters or superscripts.}

\bigskip

\begin{df} We write, for $\alpha \in K$:
\[ \| \alpha \| = k \  \Leftrightarrow \ \alpha = ut^k, \ \ u \in \oi^{\times}, k \in {\mathbb Z}. \]
If $\vec{v} \in K^n$, we write
\[ \| \vec{v} \| = k = \max \{ \ell\  | \ \vec{v} \in t^{\ell}\oi^n \} = \min \{ s \ | \  \vec{v} = t^s \vec{v}_{0}, \ \ \hbox{for some $\vec{v}_{0} \in \oi^n$} \}. \]

In the case $\alpha = 0 \in K$, we write $\| 0 \| = \infty$. We will say $\| \alpha \|$ is the \emph{norm} or the \emph{order} of $\alpha \in K$.
\end{df}

The following result is standard:

\begin{thm}
Let $V_{k}$ be a rank $k$ submodule over $\oi$ contained in $\K^n$. Then there exists a basis $\{ \vec{u}_{1}, \ldots , \vec{u}_{n} \}$ of $\oi^n$, and integers $\alpha_{1}, \ldots , \alpha_{k}$ such that
\begin{enumerate}
  \item ${\cal B}(V_{k})=\{ t^{\alpha_{1}} \vec{u}_{1}, \ldots , t^{\alpha_{k}} \vec{u}_{k} \}$ is an $\oi$-module basis of $V_{k}$.
  \item $\alpha_{1} \geq \alpha_{2} \geq \cdots \geq \alpha_{k}$.
  \item The numbers $\alpha_{1} , \ldots , \alpha_{k}$ are uniquely determined by $V_{k}$.
\end{enumerate}\label{adapted basis theorem}
\end{thm}

\begin{df} Given some rank $k$ submodule $V_{k}$, we will call ${\cal B}(V_{k})=\{ t^{\alpha_{1}} \vec{u}_{1}, \ldots , t^{\alpha_{k}} \vec{u}_{k} \}$ an \emph{invariant adapted basis} of $V_{k}$ if it satisfies the criteria given above by Theorem~\ref{adapted basis theorem}. We will call the partition $\alpha = (\alpha_{1} \geq \alpha_{2} \geq \cdots \geq \alpha_{k})$ the \emph{invariant partition} of $V_{k}$, with the entries $\alpha_{i}$ the \emph{invariant factors} of $V_{k}$, denoted
\[ inv(V_{k}) = \alpha = (\alpha_{1} \geq \alpha_{2} \geq \cdots \geq \alpha_{k}). \]
We then will write
\[ \| V_{k} \| = \alpha_{1} + \cdots + \alpha_{k}. \]

Similarly, given some $n \times k$ matrix $S$ over $K$, we shall let $\| S \|$ denote the sum of the orders of the invariant factors of the matrix $S$.

Also, given some rank-$k$ submodule $V_{k}  \subseteq \K^n$ and an invariant adapted basis ${\cal B}(V_{k}) = \{ t^{\alpha_{1}} \vec{u}_{1}, \ldots , t^{\alpha_{k}} \vec{u}_{k} \}$ for $V_{k}$, we will let, for any $1 \leq i \leq j\leq k$,
\[ ({\cal B}(V_{k}) )|^{i}_{j} = \langle \langle t^{\alpha_{i} } \vec{u}_{i}, \ldots , t^{\alpha_{j}} \vec{u}_{j} \rangle \rangle , \]
where $\langle \langle t^{\alpha_{i} } \vec{u}_{i}, \ldots , t^{\alpha_{j}} \vec{u}_{j} \rangle \rangle$ denotes the $\oi$-module generated by the elements $t^{\alpha_{i} } \vec{u}_{i}, \ldots , t^{\alpha_{j}} \vec{u}_{j}$.

\end{df}

\begin{df}
Given a pair of full-rank lattices $\la, \n \in \Gr$, as stated in the introduction, we let $Gl_{n}( \K)$ act by right-multiplication on pairs $(\n,\Lambda) \in \Gr \times \Gr$. With this action we may stratify orbits in $\Gr \times \Gr$ by setting $inv(\n,\Lambda) = \mu$ whenever $(\n,\Lambda)$ and $(I_{n} ,\m)$ are in the same $GL_{n}(\K)$ orbit, and $inv(\m) = \mu$.\label{pair invariant} \end{df}

\begin{df} Let $U_{k}$ be a rank-$k$ $\oi$-submodule of $\K^n$. Then we shall let
$\overline{U}_{k}$
denote an $n \times k$ matrix over $K$ whose columns form a basis of the $\oi$-submodule $U_{k}$. The matrix $\overline{U_{k}}$ is, therefore, not uniquely determined by $U_{k}$, but the $\oi$-span of the columns of $\overline{U_{k}}$ will be.
Given two $\oi$-modules
 $U_{k} \in \Gr^{k}$, $V_{s}\in \Gr^{s}$, we will let $[ \overline{U_{k}}|\overline{V_{s}}]$ denote the
 $n \times (k+s)$
 matrix whose columns are those of $\overline{U_{k}}$ followed by those of $\overline{V_{s}}$. When computing the orders of matrices or matrix blocks, we will generally omit the enclosing brackets, and denote the order as
 \[ \|\overline{U_{k}}|\overline{V_{s}} \|  \ \ \hbox{in place of} \ \ \| [ \overline{U_{k}}|\overline{V_{s}} ] \| \]
 though the latter is more notationally consistent.
\end{df}

With our notation, given some rank-$k$ subsmodule $U_{k} \subseteq \K^n$, then $\| U_{k} \| = \| \overline{U_{k}} \|$. That is, the sum of the orders of the invariant factors of the submodule $U_{k}$ equals that computed from a matrix of basis elements of $U_{k}$. We only necessarily have $\| U_{k} + V_{s}\| = \| \overline{U_{k}}|\overline{V_{s}}\| $ if the sum is direct, but this may fail, otherwise.

 \begin{df}
 Note that, if $\la_{t} \subset \la$ is a rank-$t$ submodule of $\la$, then there is a rank-$t$ submodule $U_{t} \subseteq \oi^n$, and some choice of bases in which the matrix equation
 \[ \overline{\la_{t}} = \overline{\la(U_{t}) } =  \overline{\la}\cdot \overline{U_{t}} \]
holds. We will generally assume our choice of matrices makes the above relations true, and refer to this as a \emph{matrix realization} of $\la_{t}$.
\end{df}

As stated in the introduction, we shall use certain invariants of submodules over $\oi$ to determine a combinatorially defined object called a hive. Hives first appeared in the work of Knutson and Tao~\cite{knut}, and their properties have been studied in many related problems~\cite{pak}.

\begin{df}A {\em hive} of size $n$ is a triangular array of
numbers $(h_{ij})_{0 \leq i \leq j \leq n}$
that satisfy the {\em rhombus inequalities}:
\begin{enumerate}
\item {\em Right-Leaning:} $h_{ij}+h_{i-1,j-1} \geq h_{i-1, j} + h_{i,j-1}$, for $1
\leq i < j \leq n$.
\item {\em Left-Leaning:} $h_{i,j} + h_{i,j-1} \geq h_{i-1,j-1} + h_{i+1,j}$, for $1
\leq i < j \leq n$.
\item {\em Vertical:} $h_{ij} + h_{i+1,j} \geq h_{i+1,j+1} + h_{i,j-1}$, for $1
\leq i < j \leq n$.
\end{enumerate} \label{hive definition}
We define the \emph{type} of a hive $\{ h_{st} \}$ as a triple of partitions $(\mu, \nu, \lambda)$ of length $n$, where $\mu=(\mu_{1}, \ldots , \mu_{n})$ gives the differences down the left edge of the hive, $\nu$ gives the differences along the bottom, and $\lambda$ gives the differences along the bottom, where we set $h_{00} = 0$. Specifically,

\begin{align*}
\mu_i &= h_{0,i} - h_{0,(i-1)} \quad & \hbox{(the downward
differences of entries
along the left side)} \\
\nu_i &= h_{i,n} - h_{(i-1),n} \quad & \hbox{(the rightward
differences of entries along the bottom)}\\ \lambda_{i} &=
 h_{ii} -h_{(i-1)(i-1)}  \quad & \hbox{(the downward differences of
entries along the right side)}.
\end{align*}
It is a consequence of the rhombus inequalities that these numbers form non-increasing partitions $\mu,\nu$ and $\lambda$, where $\mu = (\mu_{1} \geq \mu_{2} \geq \cdots \geq \mu_{n})$, $\nu = (\nu_{1} \geq \nu_{2} \geq \cdots \geq \nu_{n})$, and $\lambda = (\lambda_{1} \geq \lambda_{2} \geq \cdots \geq \lambda_{n})$.
\end{df}

A hive of size 4 is shown below:
\[ \setlength{\arraycolsep}{-0.2mm} \renewcommand{\arraystretch}{1.7}
\begin{array}{ccccccccccc}
    &     &    &    &     &h_{00}&     &     &      &   &  \\
    &     &    &    &h_{01}&     &h_{11}&     &      &   &  \\
    &     &    &h_{02}&     &h_{12}&     &h_{22}&      &   &  \\
    &     &h_{03}&    &h_{13}&     &h_{23}&     &h_{33}&   &  \\
    &h_{04}&    &h_{14}&     &h_{24}&     &h_{34}&      &h_{44}&  \\
%h_{05} & & h_{15} & & h_{25} & & h_{35} & & h_{45} && h_{55}
 \end{array}  \hskip .25in=\hskip .25in
 \setlength{\arraycolsep}{-0.2mm} \renewcommand{\arraystretch}{1.7}
\begin{array}{ccccccccccc}
    &     &    &    &     &\,0\,&     &     &      &   &  \\
    &     &    &    &\,21\,&     &\,27\,&     &      &   &  \\
    &     &    &\,34\,&     &\,44\,&     &\,48\,&      &   &  \\
    &     &\,40\,&    &\,54\,&     &\,64\,&     &\,67\,&   &  \\
    &41&    &58&     &72&     &81&      &83&  \\
%h_{05} & & h_{15} & & h_{25} & & h_{35} & & h_{45} && h_{55}
 \end{array} \]
 Its type is $(\mu,\nu,\lambda)$ where $\mu = (21,13,6,1)$, $\nu = (17,14,9,2)$, and $\lambda = (27,21,19,16)$.

The inequalities above are so named because their entries correspond to rhombi formed by adjacent
entries in the array such that the upper acute angle points to the right,
vertically, and to the left, respectively.  In each case the inequality asserts
that the sum of the entries of the obtuse vertices of the rhombus is greater
than or equal to the sum of the acute entries.

 \section{Technical Lemmas}

 We begin with some technical results that will allow us to write matrix decompositions in a somewhat simpler form.
 \begin{lem}Let $\la^{(1)}, \la^{(2)}, \ldots , \la^{(s)}$ be full-rank lattices in $\K^n$, and $\la^{(i)}_{a_{i}} \subseteq \la^{(i)}$ be a submodule of rank $a_{i}$. We assume the sum
  \[ \la^{(1)}_{a_{1}}  \oplus \la^{(2)}_{a_{2}} \oplus \cdots \oplus \la^{(s)}_{a_{s}}  \]
  is direct.

  We choose submodules of $\oi^n \subseteq \K^n$, denoted $U_{a_{1}}, U_{a_{2}}, \ldots , U_{a_{s}}$, of ranks $a_{1}, a_{2} \ldots, a_{s}$ yielding some matrix realization
  \[ \overline{\la^{(i)}_{a_{i}} } = \overline{ \la^{(i)}(U_{a_{i}})}. \]
 Then there also exists a matrix realization of these submodules that have a block upper triangular decomposition:
 \[ \left[
 \overline{ \la_{1}(U_{a_{1}})} |\overline{ \la_{2}(U_{a_{2}})} | \cdots |\overline{ \la_{s}(U_{a_{s}})} \right] =
 \left[ \begin{array}{c|c|c|c}
\left(  \overline{ \la_{1}(U_{a_{1}})}\right)^{(1)} & \left(\overline{ \la_{2}(U_{a_{2}})}\right)^{(1)} &  \cdots &\left(\overline{ \la_{s}(U_{a_{s}})} \right)^{(1)} \\
0 & \left(\overline{ \la_{2}(U_{a_{2}})}\right)^{(2)} &  \cdots &\left(\overline{ \la_{s}(U_{a_{s}})} \right)^{(2)} \\
\vdots & \ddots & \ddots & \vdots \\
0 & \cdots & 0 &\left(\overline{ \la_{s}(U_{a_{s}})} \right)^{(s)} \end{array} \right] \]
where each block on the diagonal
\[ \left(\overline{ \la_{k}(U_{a_{k}})} \right)^{(k)} \]
is of size $a_{k} \times a_{k}$, is itself a diagonal matrix $diag(t^{\beta^{k}_{1}}, t^{\beta^{k}_{2}},  \ldots , t^{\beta^{k}_{a_{k}}}))$ such that $\beta^{k}_{1} \geq \beta^{k}_{2} \geq \cdots \beta^{k}_{a_{k}}$, and, in particular, we have:
\[ \|  \overline{ \la_{1}(U_{a_{1}})} |\overline{ \la_{2}(U_{a_{2}})} | \cdots |\overline{ \la_{s}(U_{a_{s}})} \| = \left\|  \left(\overline{ \la_{1}(U_{a_{1}})}\right)^{(1)} \right\| +\left\|  \left(\overline{ \la_{1}(U_{a_{1}})}\right)^{(2)} \right\| + \cdots + \left\| \left( \overline{ \la_{1}(U_{a_{1}})}\right)^{(s)} \right\|. \] \label{normal}
 \end{lem}
\begin{df} A matrix realization of submodules in the form above, satisfying the conclusions of Lemma~\ref{normal}, will be said to be in \emph{normal form}. \end{df}

\begin{proof}
By row operations and column operations in the first $a_{1}$-many columns only, we may assume the matrix $ \overline{ \la_{1}(U_{a_{1}})}$ is in normal form. We then continue, using only row operations in rows $a_{1} + 1$ to $n$, and in columns $a_{1} +1$ to $a_{1} + a_{2}$, etc., proceeding down the block diagonal at each stage.
\end{proof}

 Speyer's work on hives and Vinnikov curves~\cite{speyer} explored, among other things, the relation between hive constructions and some aspects of tropical mathematics (where multiplication and addition of two numbers is replaced by the addition and the minimum of two numbers, respectively). In our setting, we are able to realize this relationship rather explicitly in the form of Lemma~\ref{amalgam} below, which gives a precise characterization of the submodules at which the minima or maxima appearing in Theorem~\ref{main theorem} above may be realized.

 Lemma~\ref{amalgam} will be proved for two \emph{arbitrary} full rank $\oi$-lattices ${\cal A}$ and ${\cal B}$ in $K^n$. In this way we will be able to apply it to \emph{both} Equation~\eqref{mu,nu,lambda min}, producing a hive of type $(\mu,\nu, \lambda)$ and also Equation~\eqref{nu,mu,lambda min}, producing a hive of type $(\nu, \mu, \lambda)$. The characterization given below will be used in the following section to show that our construction satisfies the rhombus inequalities (Theorem~\ref{min formula}).

 \begin{lem} Let ${\cal A}, {\cal C} \in \Gr$ be two full $\oi$-lattices in $K^n$. For all $s,t \in {\mathbb N}$ where $s+t \leq n$, let us define
 \[ M_{st} = \min \{ \| {\cal A}_{s} \oplus {\cal C}_{t} \| : {\cal A}_{s} \subseteq {\cal A}, \ \ {\cal C}_{t} \subseteq {\cal C} \}. \]
 Then,
 given any two values $M_{st}$ and $M_{s't'}$ defined as above, if $t \leq t'$, we may assume that the minima may be realized at submodules $A_{s}, C_{t}$ and $A_{s'}, C_{t'}$, respectively, such that $C_{t} = C_{t'}\!|^{t'-t+1}_{t'}$. In particular,  we may require $C_{t} \subseteq C_{t'}$, and if $t = t'$, that $C_{t} = C_{t'}$.
\label{amalgam}
 \end{lem}

 \begin{proof}
 Suppose $M_{st}$ and $M_{s't'}$ are realized at submodules $A_{s}, C_{t}$ and $A_{s'}, C_{t'}$, respectively, and suppose $t \leq t'$. Let
 \[ C_{t}= \langle \langle c_{1} , \ldots , c_{t} \rangle \rangle, \ \ \ C_{t'}= \langle \langle c_{1}' , \ldots , c_{t'}' \rangle \rangle \]
 denote the span of bases for $C_{t}$ and $C_{t'}$, respectively. We will now define generators for submodules $C_{t}^*$ and $C_{t'}^*$ such that
 \[ M_{st} =\| {\cal A}_{s} \oplus {\cal C}_{t} \| = \| {\cal A}_{s} \oplus {\cal C}_{t}^* \| \ \ \hbox{and} \ \ M_{s't'} =\| {\cal A}_{s'} \oplus {\cal C}_{t'} \| = \| {\cal A}_{s'} \oplus {\cal C}_{t'}^* \|,  \]

 where
 \[ {\cal C}_{t'}^* = \langle \langle c_{1}^*, \ldots , c_{t'}^* \rangle \rangle \]
 and
 \[ {\cal C}_{t}^* = \langle \langle c_{1}, \ldots, c_{t'-t},c_{1}^*, \ldots ,  , c_{t}^* \rangle \rangle. \]

 We start at the index $t'$, and proceed with smaller values of the index. There are three cases:

 \begin{enumerate}
 \item If replacing $c_{t}$ with $c'_{t'}$ in $C_{t}$ does not change the norms of the associated submodules, that is, if
 \[ \| {\cal A}_{s} \oplus {\cal C}_{t} \| = \| {\cal A}_{s} \oplus \langle \langle c_{1}, \ldots, c_{t-1},c_{t'} \rangle \rangle \|, \]
 then set $c_{t}^* = c_{t'}^* = c_{t'}'. $
 \item If $\| {\cal A}_{s} \oplus {\cal C}_{t} \| \neq \| {\cal A}_{s} \oplus \langle \langle c_{1}, \ldots, c_{t-1},c_{t'} \rangle \rangle \|$, but we may replace $c'_{t'}$ with $c_{t}$, so that
 \[ \| {\cal A}_{s'} \oplus {\cal C}_{t'} \| = \| {\cal A}_{s'} \oplus \langle \langle c_{1}', \ldots, c_{t'-1}',c_{t} \rangle \rangle \|, \]
 then set $c_{t}^* = c_{t'}^* = c_{t}. $
 \item If the previous two steps do not apply, then we may assume
 \[ \| {\cal A}_{s} \oplus {\cal C}_{t} \| < \| {\cal A}_{s} \oplus \langle \langle c_{1}, \ldots, c_{t-1},c_{t'}' \rangle \rangle \| \]
 and
 \[ \| {\cal A}_{s'} \oplus {\cal C}_{t'} \| < \| {\cal A}_{s'} \oplus \langle \langle c_{1}', \ldots, c_{t'-1}',c_{t} \rangle \rangle \| \]
 by the definition of $M_{st},M_{s', t'}$ as the \emph{minimum} among all submodules of the appropriate ranks. In this case, we set:
 \[ c_{t}^* = c_{t'}^* =c_{t'}'+ c_{t} \]
 so that, by properties of determinant (applied to any matrix realization of these generators as columns of matrices) we have:
 \begin{align*} \| {\cal A}_{s'} \oplus \langle \langle c_{1}', \ldots c_{t'-1}', c_{t'}'+ c_{t} \rangle \rangle \| =& \| {\cal A}_{s'} \oplus \langle \langle c_{1}', \ldots c_{t'-1}', c_{t'}' \rangle \rangle \| \\
 = & \| {\cal A}_{s'} \oplus {\cal C}_{t'} \| \end{align*}
 and similarly
 \[ \| {\cal A}_{s} \oplus \langle \langle c_{1}, \ldots c_{t-1}, c_{t'}'+ c_{t} \rangle \rangle \|
 =  \| {\cal A}_{s} \oplus {\cal C}_{t} \| . \]
 \end{enumerate}

 Repeating this construction with the generators $c'_{t'-1}, c'_{t'-2}$, etc., allows us to conclude, ultimately, that $C_{t} \subset C_{t'}$. In particular, in the case $t = t'$, we see that the generators for ${\cal C}_{t}$ and ${\cal C}_{t'}$ are the same.

 To prove that the minima are achieved with $C_{t}$ and $C_{t'}$ such that $C_{t} = C_{t'}\!|^{t'-t+1}_{t'}$ we now proceed matricially, and choose matrix realizations $\overline{{\cal A}_{s}}, \overline{{\cal C}_{t}}, \overline{{\cal A}_{s'}}$ and $\overline{{\cal C}_{t'}}$ for the modules ${\cal A}_{s}, {\cal C}_{t}, {\cal A}_{s'}$ and ${\cal C}_{t'}$, respectively, where we may now assume ${\cal C}_{t} \subset {\cal C}_{t'}$, implying that the columns of the $n \times t$ matrix $\overline{{\cal C}_{t}}$ are in the span of the columns of $\overline{{\cal C}_{t'}}$.

 Then, by standard arguments we may find an invertible $P \in GL_{n}(\oi)$ and $Q_{t'} \in GL_{t'}(\oi)$ such that we have the normal form:
 \[ P [ \overline{{\cal C}_{t'}} \, | \, \overline{{\cal A}_{s'}}] \begin{bmatrix} Q & 0 \\ 0 & I_{s'} \end{bmatrix} =  \left[ \begin{array}{cccc|ccc}
 t^{\beta_1} & 0 & \cdots &0 & && \\
 0 & t^{\beta_2} & \ddots & \vdots & & \left(P\overline{{\cal A}_{s'}}\right)^{(1)}&\\ \vdots & \ddots & \ddots & 0 & && \\
 0 & \cdots & 0 & t^{\beta_{t'}} & && \\ \hline
 & & \!\!\!\!\!\!\!\!\!\!\!\!{\Huge 0} & &  & \left(P\overline{{\cal A}_{s'}}\right)^{(2)} & \\
 \end{array} \right] \]
 where the lower left ``$0$'' denotes an $(n-t') \times t'$ block of zeroes, and the sequence $(\beta_{1} \geq \beta_{2} \geq \cdots \geq \beta_{t'}) = inv({\cal C}_{t'})$.  We will denote the columns of $C_{t'}$ above by $c'_{1}, \ldots , c'_{t'}$. Since the columns of $\overline{{\cal C}_{t}}$ are in the span of the columns of $\overline{{\cal C}_{t'}}$, applying the same row transformations to $\overline{{\cal C}_{t}}$ will yield:

 \[ P [ \overline{{\cal C}_{t}} \, | \, \overline{{\cal A}_{s}}]  =
 \left[ \begin{array}{c|c} C_t & \left(P\overline{{\cal A}_{s}}\right)^{(1)}\\ \hline 0 &   \left(P\overline{{\cal A}_{s}}\right)^{(2)}\end{array} \right]  \]
 where $C_{t}$ denotes a $t' \times t$ matrix in the span the diagonal matrix above. Applying further column transformations to $C_{t}$ will allow us to conclude that the columns of $C_{t}$ form an invariant adapted basis for the transformed image of ${\cal C}_{t}$. That is, we may assume
 \[ C_{t} = [ t^{\alpha_{1} } \vec{u}_{1}, \ldots , t^{\alpha_{t}} \vec{u}_{t} ] \]
 where $inv(C_{t}) = inv(\overline{{\cal C}_{t}}) = (\alpha_{1} \geq \cdots \geq \alpha_{t})$, and $[ \vec{u}_{1}, \ldots , \vec{u}_{t} ]$ is a $t' \times t$ matrix such that $\| \vec{u}_{i} \| = 0$ for $1 \leq i \leq t$.

We will now start in column $t$, and systematically work on the columns of $C_{t}$ to ensure:

\begin{enumerate}
\item The columns $c'_{1}, \ldots, c'_{t'}$ and $u_{1}t^{\alpha_{1}},\ldots , u_{t}t^{\alpha_{t}}$ form invariant-adapted bases of $C_{t'}$ and $C_{t}$, respectively.
\item $c_{t'-k} = u_{t-k}t^{\alpha_{t-k}}$, for $k = 0, \ldots, t-1$.
\end{enumerate}

In column $t$, if we have $\alpha_{t} = \beta_{t'}$, then either we are done in this column (meaning $c_{t'} = u_{t}t^{\alpha_{t}}$), or, there is an entry of order $\alpha_{t}$ in the column $u_{t}t^{\alpha_{t}}$ in a row $s$ higher than row $t'$. This implies that the invariant factors of $C_{t'}$ satisfy $\beta_{s} = \beta_{s+1} = \cdots = \beta_{t'}$. Consequently, we may swap some rows (on all matrices) and columns of $C_{t'}$ so that $C_{t'}$ is still in diagonal form, and the entry in row $t'$ of $C_{t}$ has order $\beta_{t'}$, which is necessarily minimal among the orders of row $t'$ of $C_{t}$. Then, in either case we may, by column operations on $C_{t}$, ensure that all entries in row $t'$ in columns 1 through $t-1$ of $C_{t}$ are now $0$.

If $\alpha_{t} \neq \beta_{t'}$, we must have $\beta_{t'} < \alpha_{t}$, since ${\cal C}_{t} \subset {\cal C}_{t'}$, and $\beta_{t'}$ is minimal among all orders of elements in the submodule. In fact, we must have $\beta_{t'}$ is less than the orders of all entries in column $t$ of $C_{t}$.  We will argue in a manner similar to the above. If
\[ \| t^{\alpha_{1} } \vec{u}_{1}, \ldots , t^{\alpha_{t-1}} \vec{u}_{t-1}, c_{t'}| A_{s} \| = \| t^{\alpha_{1} } \vec{u}_{1}, \ldots , t^{\alpha_{t}} \vec{u}_{t} |A_{s} \|, \]
then we may replace $ t^{\alpha_{t}} \vec{u}_{t}$ in $C_{t}$ above with $c'_{t'}$. Otherwise, we have
\[ \| t^{\alpha_{1} } \vec{u}_{1}, \ldots , t^{\alpha_{t-1}} \vec{u}_{t-1}, c'_{t'}| A_{s} \| > \| t^{\alpha_{1} } \vec{u}_{1}, \ldots , t^{\alpha_{t}} \vec{u}_{t} |A_{s} \|, \]
since the right member of the inequality has minimal order among all submodules of appropriate ranks. In this case we may replace \emph{both} columns $c'_{t'}$ and $ t^{\alpha_{t}} \vec{u}_{t}$ with their sum $c'_{t'} +  t^{\alpha_{t}} \vec{u}_{t}$, noting that by the above inequality, we must have
\[ \| t^{\alpha_{1} } \vec{u}_{1}, \ldots ,(c'_{t'}+  t^{\alpha_{t-1}} \vec{u}_{t-1})| A_{s} \| = \| t^{\alpha_{1} } \vec{u}_{1}, \ldots , t^{\alpha_{t}} \vec{u}_{t} |A_{s} \|, \]
and also
\[ \| c'_{1}, \ldots , c'_{t'-1}, (c'_{t'}+  t^{\alpha_{t-1}} \vec{u}_{t-1}) | A_{s'} \| =  \| c'_{1}, \ldots , c'_{t'-1}, c'_{t'}  | A_{s'} \| \]
since the columns $c'_{1}, \ldots , c'_{t'-1}, (c'_{t'}+  t^{\alpha_{t-1}} \vec{u}_{t-1})$ equal the span of $c'_{1}, \ldots , c'_{t'-1}, c'_{t'} $. In either case, we may then assume that the entry in row $t'$ of column $t$ of $C_{t}$ is of (minimal) order $\beta_{t'}$, and hence we may (by column operations) ensure that all entries in $C_{t}$ in row $t'$, columns $1$ through $t-1$, are zero.
Further, in all cases, we may ensure that column $t$ of $C_{t}$ and column $t'$ of $C_{t'}$ equal, and are the columns of an invariant adapted basis corresponding to the smallest invariant factor.

 We then argue in precisely the same manner in column $t-1$, proceeding through all succeeding columns.
 \end{proof}

 We note also that the hypotheses of the above lemma are, up to ordering of notation, symmetric in the modules ${\cal A}$ and ${\cal C}$.

\section{Satisfying the Rhombus Inequalities}

We fix some notation for the rest of the paper. Let $\la,\n \in \Gr$ be two full-rank lattices, with $inv(\la) = \lambda$, $inv(\n) = \nu$, and suppose $(\n, \la)$ is in the same $\K$-orbit as $(I,\m)$, so that we have, by Definition~\ref{pair invariant}, $inv(\n, \la) = inv(\m) = \mu$. In particular, under any matrix identification of these lattices, we have $\la = \n \m$ (as a product of $n \times n$ matrices over $\K$ of full rank).

\begin{thm}
     Choose $\Lambda,N \in \Gr$. Let the invariant partition of $\la$ be $inv(\Lambda) = \lambda = (\lambda_{1} , \ldots , \lambda_{n})$, and let $| \lambda | = \lambda_{1} + \cdots + \lambda_n$.

      Setting
\begin{align}   h_{st} &=  \  | \lambda | - \min_{ \Lambda_{n-t} \oplus {\scriptstyle{\cal N}_{t-s}}} \!\!\big(  \left\| \Lambda_{n-t}\oplus \n_{t-s} \right\| \big)
     \label{min formula} \end{align}
     the collection of numbers $\{ h_{ij} \}$ satisfies the rhombus inequalities found in Definition~\ref{hive definition}, and so forms a hive. \label{rhombus}
\end{thm}

\begin{proof}

We must prove that the numbers $\{ h_{ij} \}$ given by Equation~\eqref{min formula} satisfy the rhombus inequalities:

\begin{enumerate}
\item {\em Right-Leaning:} $h_{ij}+h_{i-1,j-1} \geq h_{i-1, j} + h_{i,j-1}$, for $1
\leq i < j \leq n$.
\item {\em Left-Leaning:} $h_{i,j} + h_{i,j-1} \geq h_{i-1,j-1} + h_{i+1,j}$, for $1
\leq i < j \leq n$.
\item {\em Vertical:} $h_{ij} + h_{i+1,j} \geq h_{i+1,j+1} + h_{i,j-1}$, for $1
\leq i < j \leq n$.
\end{enumerate}
All three inequalities are proved in essentially the same way, depending quite explicitly on the characterization given by Lemma~\ref{amalgam} for the minima appearing in Equation~\eqref{min formula}, but depending on slightly different interlacing inequalities in each case.

\bigskip
\bigskip
\noindent \underline{Right-Leaning Rhombus Inequality for the $\{ h_{ij} \}$.}

Let us suppose the minima given by Equation~\eqref{min formula} are realized at specific submodules:

\begin{align*} h_{ij} & =   \  | \lambda | - \min_{ \Lambda_{n-j} \oplus {\scriptstyle{\cal N}_{j-i}}} \!\!\big(  \left\| \la_{n-j}  \oplus  \n_{j-i} \right\| \big) \\
& = | \lambda | - \left(  \left\| \la_{n-j}^{(ij)}  \oplus  \n_{j-i}^{(ij)} \right\| \right) ,
\end{align*}
where the superscripts will denote the indices of the proposed hive entry to which it corresponds, and the subscripts will denote, as always, the ranks of the submodules.

%\begin{align*}   A:& \ \ \| MN(V_{n-j}^{(A)}) \|  + \| N(W_{j-i}^{(A)}) \|  =& \min_{\substack{ V_{n-j} \oplus W_{j-i}}}  \big( \| MN(V_{n-j})\| + \| N(W_{j-i})  \| \big) \\
%B:& \ \ \| MN(V_{n-j+1}^{((i-1),(j-1))}) \| +  \| N(W_{j-i}^{((i-1),(j-1))}) \|
% =& \min_{\substack{ V_{n-j+1} \oplus W_{j-i}}}  \big( \| MN(V_{n-j+1}) \|  + \| N(W_{j-i}) \| \big) \\
% C:& \ \ \| MN(V_{n-j}^{(C)}) \|  + \|  N(W_{j-i+1}^{(C)}) \|
 % =& \min_{\substack{ V_{n-j} \oplus W_{j-i+1}}}  \big( \| MN(V_{n-j}) \|   + \| N(W_{j-i+1}) \| \big)  \\
%  D:& \ \ \| MN(V_{n-j+1}^{(D)}) \|  + \|  N(W_{j-i-1}^{(D)}) \|
% =& \min_{\substack{ V_{n-j+1} \oplus W_{j-i-1}}} \big( \| MN(V_{n-j+1}) \|  + \| N(W_{j-i-1}) \|  \big)
%\end{align*}

We use this to replace each entry in the right-leaning rhombus inequality:
\[ h_{ij}+h_{i-1,j-1} \geq h_{i-1, j} + h_{i,j-1}. \]

We may then re-write the above (after subtracting all the constants $| \lambda |$ appearing on both sides), and are left with proving:

\[ \left\| \la_{n-j}^{(ij)}  \oplus  \n_{j-i}^{(ij)} \right\|  +
 \left\| \la_{n-j+1}^{(i-1,j-1)}  \oplus  \n_{j-i}^{(i-1,j-1)} \right\|
 \leq
 \left\| \la_{n-j}^{(i-1,j)}  \oplus  \n_{j-i+1}^{(i-1,j)} \right\|  +
\left\| \la_{n-j+1}^{(i,j-1)}  \oplus  \n_{j-i-1}^{(i,j-1)} \right\| \]
which we shall re-write as:

\begin{multline}
 \left\| \la_{n-j+1}^{(i-1,j-1)}  \oplus  \n_{j-i}^{(i-1,j-1)} \right\|
 - \left\| \la_{n-j+1}^{(i,j-1)}  \oplus  \n_{j-i-1}^{(i,j-1)} \right\|
 \leq \\
 \left\| \la_{n-j}^{(i-1,j)}  \oplus  \n_{j-i+1}^{(i-1,j)} \right\|
  -\left\| \la_{n-j}^{(ij)}  \oplus  \n_{j-i}^{(ij)} \right\| . \label{main} \end{multline}

We shall now perform a series of substitutions to the modules appearing in the above inequality, using Lemma~\ref{amalgam} and other arguments, which will imply Inequality~\ref{main}. To start, by Lemma~\ref{amalgam} we may replace the submodules $\n_{j-i}^{(i-1,j-1)}$ and $ \n_{j-i-1}^{(i,j-1)}$, appearing on the left side of Inequality~\ref{main}, with submodules $\n_{j-i}^{A}$ and $ \n_{j-i-1}^{A}$ such that

\begin{equation} \n_{j-i-1}^{A}= \left(\n_{j-i}^{A}\right)\!|^{2}_{j-i}, \label{condition A} \end{equation}
while maintaining the equalities
\[  \left\| \la_{n-j+1}^{(i-1,j-1)}  \oplus  \n_{j-i}^{A} \right\| =  \left\| \la_{n-j+1}^{(i-1,j-1)}  \oplus  \n_{j-i}^{(i-1,j-1)} \right\| \]
and
\[  \left\| \la_{n-j+1}^{(i,j-1)}  \oplus  \n_{j-i-1}^{A} \right\| =  \left\| \la_{n-j+1}^{(i,j-1)}  \oplus  \n_{j-i-1}^{(i,j-1)} \right\|. \]

In similar fashion, using Lemma~\ref{amalgam}, we will replace the modules $\la_{n-j+1}^{(i-1,j-1)}$ and $ \la_{n-j}^{(i-1,j)}$ (the left-most summands appearing in both the left and right members of Inequality~\ref{main}) with $\la_{n-j+1}^{B}$ and $ \la_{n-j}^{B}$ such that
\begin{equation} \la_{n-j}^{B}= \left(\la_{n-j+1}^{B}\right)\!|^{2}_{n-j+1}, \label{condition B} \end{equation}
while maintaining the equalities
\[ \left\| \la_{n-j+1}^{B}  \oplus  \n_{j-i}^{A} \right\| = \left\| \la_{n-j+1}^{(i-1,j-1)}  \oplus  \n_{j-i}^{A} \right\| \]
and
\[ \left\| \la_{n-j}^{B}  \oplus  \n_{j-i+1}^{(i-1,j)} \right\| = \left\| \la_{n-j}^{(i-1,j)}  \oplus  \n_{j-i+1}^{(i-1,j)} \right\|. \]
Thirdly, we will now replace the submodules $ \n_{j-i+1}^{(i-1,j)} $ and $ \n_{j-i}^{(ij)}$, appearing as the two right-hand summands in the right member of Inequality~\ref{main}, with submodules $ \n_{j-i+1}^{C} $ and $ \n_{j-i}^{C}$ such that
\begin{equation}  \n_{j-i}^{C}= \left(\n_{j-i+1}^{C} \right)\!|^{2}_{n-j+1}, \label{condition C} \end{equation}
while maintaining the equalities
\[  \left\| \la_{n-j}^{B}  \oplus  \n_{j-i+1}^{C} \right\| =  \left\| \la_{n-j}^{(i-1,j)}  \oplus  \n_{j-i+1}^{C} \right\| \]
and
\[ \left\| \la_{n-j}^{(ij)}  \oplus  \n_{j-i}^{(ij)} \right\| = \left\| \la_{n-j}^{(ij)}  \oplus  \n_{j-i}^{(ij)} \right\|. \]

Making these replacements, we see that proving Inequality~\ref{main} above is equivalent to proving
\begin{equation}
 \left\| \la_{n-j+1}^{B}  \oplus  \n_{j-i}^{A} \right\|
 - \left\| \la_{n-j+1}^{(i,j-1)}  \oplus  \n_{j-i-1}^{A} \right\|
 \leq
 \left\| \la_{n-j}^{B}  \oplus  \n_{j-i+1}^{C} \right\|
  -\left\| \la_{n-j}^{(ij)}  \oplus  \n_{j-i}^{C} \right\|, \label{main ABC} \end{equation}
 subject to the conditions~\ref{condition A},~\ref{condition B}, and~\ref{condition C} above.

 We now proceed somewhat differently. Let us first consider $\la_{n-j}^{B}  \oplus  \n_{j-i+1}^{C}$ and $ \la_{n-j}^{(ij)}  \oplus  \n_{j-i}^{C}$, the submodules appearing in the right member of Inequality~\ref{main ABC}. Since the sum $\la_{n-j}^{B}  \oplus  \n_{j-i+1}^{C}$ is direct, we may conclude (by Condition~\ref{condition C}) that the sum $ \la_{n-j}^{B}  \oplus  \n_{j-i}^{C}$ is direct as well. However, since $\left\| \la_{n-j}^{(ij)}  \oplus  \n_{j-i}^{C} \right\|$ is the minimal value among such summands, we have
 \[ \left\| \la_{n-j}^{(ij)}  \oplus  \n_{j-i}^{C} \right\| \leq \left\| \la_{n-j}^{B}  \oplus  \n_{j-i}^{C} \right\| \]
 and consequently
 \[  \left\| \la_{n-j}^{B}  \oplus  \n_{j-i+1}^{C} \right\|
  -\left\| \la_{n-j}^{B}  \oplus  \n_{j-i}^{C} \right\| \leq
   \left\| \la_{n-j}^{B}  \oplus  \n_{j-i+1}^{C} \right\|
  -\left\| \la_{n-j}^{(ij)}  \oplus  \n_{j-i}^{C} \right\|, \]
 so that Inequality~\ref{main  ABC} above will be implied by
 \begin{equation}
 \left\| \la_{n-j+1}^{B}  \oplus  \n_{j-i}^{A} \right\|
 - \left\| \la_{n-j+1}^{(i,j-1)}  \oplus  \n_{j-i-1}^{A} \right\|
 \leq
 \left\| \la_{n-j}^{B}  \oplus  \n_{j-i+1}^{C} \right\|
  -\left\| \la_{n-j}^{B}  \oplus  \n_{j-i}^{C} \right\| \label{main ABBC}. \end{equation}

To proceed we now work matricially. Let us use matrix representations $\left[ \overline{\la_{n-j+1}^{B}}\right], \left[\overline{\n_{j-i}^{A}}\right]$, etc., in the above, where, for example, $\left[\overline{\la_{n-j+1}^{B}}\right]$ denotes an $n \times (n-j+1)$ matrix whose columns span the submodule $\la_{n-j+1}^{B}$. Thus, we may re-express Inequality~\ref{main ABBC} as:
\[
 \left\| \overline{\la_{n-j+1}^{B}}  | \overline{ \n_{j-i}^{A}}  \right\|
 - \left\| \overline{\la_{n-j+1}^{(i,j-1)}}  | \overline{ \n_{j-i-1}^{A}} \right\|
 \leq
 \left\| \overline{\la_{n-j}^{B}}  | \overline{ \n_{j-i+1}^{C}} \right\|
  -\left\| \overline{\la_{n-j}^{B}}  |  \overline{\n_{j-i}^{C}} \right\| \]

We claim, given modules (column vectors) $ \overline{\la_{n-j+1}^{B}}$, $ \overline{ \n_{j-i}^{A}} $,$ \overline{\la_{n-j+1}^{(i,j-1)}} $, $ \overline{ \n_{j-i-1}^{A}} $ above, that
\[ \left\| \overline{\la_{n-j+1}^{(i,j-1)}}  | \overline{ \n_{j-i-1}^{A}} \right\| = \left\| \overline{\la_{n-j+1}^B}  | \overline{ \n_{j-i-1}^{A}} \right\| . \]

The matrix $ \left[ \overline{\la_{n-j+1}^B}  | \overline{ \n_{j-i}^{A}} \right]$ is is composed of the left hand block $\overline{\la_{n-j+1}^B} $, whose columns form an $n \times (n-j+1)$ matrix, and then the block $\overline{ \n_{j-i}^{A}}$, forming an $n \times (j-i+1)$ matrix. Thus, there is an invertible $n \times n$ matrix $P$ of row operations, and a square matrix $Q$ of column operations with $(n-j+1) + (j-i)$ rows such that
\[ P\left[  \overline{\la_{n-j+1}^B}  | \overline{ \n_{j-i}^{A}} \right] Q = \left[ \begin{array}{c|c} \left(\overline{\la_{n-j+1}^B} \right)^{(1)} &  \left(\overline{ \n_{j-i}^{A}} \right)^{(1)}  \\ \hline 0 & \left(\overline{ \n_{j-i}^{A}} \right)^{(2)} \\ \hline 0 & 0 \end{array} \right] \]
where $\left(\overline{\la_{n-j+1}^B} \right)^{(1)}$ is a square matrix with $n-j+1$ rows, and $\left(\overline{ \n_{j-i}^{A}} \right)^{(2)}$ is a square matrix with $j-i$ rows, and all the $0$'s denote matrices of zeros of appropriate size. By Equation~\ref{condition A} we have
\[ \n_{j-i-1}^{A}= \left(\n_{j-i}^{A}\right)\!|^{2}_{j-i}, \]
so that we may actually assume
\[ P\left[  \overline{\la_{n-j+1}^B}  | \overline{ \n_{j-i}^{A}} \right] Q = \left[ \begin{array}{c|c} \left(\overline{\la_{n-j+1}^B} \right)^{(1)} &  \left(\overline{ \n_{j-i}^{A}} \right)^{(1)}  \\ \hline 0 & \begin{array}{c|c} t^{\beta_{1}} & 0 \\ \hline 0 & \left(\overline{ \n_{j-i-1}^{A}} \right)^{(2)} \end{array} \\ \hline 0 & 0 \end{array} \right], \]
where $\beta_{1}$ is the order of the largest invariant factor of $\left(\overline{ \n_{j-i}^{A}} \right)^{(2)}$.

We note that we may multiply all the matrices in Inequality~\ref{main ABBC} on the left by the invertible matrix $P$ without changing the orders of the invariants in any terms. Further, arbitrary column operations within each matrix block \emph{separately} are permissible since we are only computing the orders of the invariant factors of the modules that the columns span. Thus, we may simultaneously assume that we have both

\[ \left[  \overline{\la_{n-j+1}^B}  | \overline{ \n_{j-i}^{A}} \right]  = \left[ \begin{array}{c|c} \left(\overline{\la_{n-j+1}^B} \right)^{(1)} &  \left(\overline{ \n_{j-i}^{A}} \right)^{(1)}  \\ \hline 0 & \begin{array}{c|c} t^{\beta_{1}} & 0 \\ \hline 0 & \left(\overline{ \n_{j-i-1}^{A}} \right)^{(2)} \end{array} \\ \hline 0 & 0 \end{array} \right], \]
and also
\[ \left[   \overline{\la_{n-j+1}^{(i,j-1)}}  | \overline{ \n_{j-i-1}^{A}} \right]  = \left[ \begin{array}{c|c} \left( \overline{\la_{n-j+1}^{(i,j-1)}}  \right)^{(1)} &  \left(\overline{ \n_{j-i}^{A}} \right)^{(1)}  \\ \hline  \left( \overline{\la_{n-j+1}^{(i,j-1)}}  \right)^{(2)} & \begin{array}{c} 0 \\ \hline  \left(\overline{ \n_{j-i-1}^{A}} \right)^{(2)} \end{array} \\ \hline  \left( \overline{\la_{n-j+1}^{(i,j-1)}}  \right)^{(3)} & 0 \end{array} \right]. \]

If a row operation adds any row in the block $ \left( \overline{\la_{n-j+1}^{(i,j-1)}}  \right)^{(3)}$ \emph{upwards} into the rows of the block $\left( \overline{\la_{n-j+1}^{(i,j-1)}}  \right)^{(1)}$, or even in the top row of the block $ \left(  \overline{\la_{n-j+1}^{(i,j-1)}}  \right)^{(2)}$ (adding similarly upwards), then not only will the form of the matrix $\left[   \overline{\la_{n-j+1}^{(i,j-1)}}  | \overline{ \n_{j-i-1}^{A}} \right]$ be preserved (preserving the blocks of zeros on the right in $\left[ \overline{ \n_{j-i-1}^{A}} \right]$) but it will \emph{also} fix the matrix $ \left[  \overline{\la_{n-j+1}^B} \right]$. Thus, we may assume both
\[ \left\| \overline{\la_{n-j+1}^B}  | \overline{ \n_{j-i-}^{A}} \right\| = \left\|  \left(\overline{\la_{n-j+1}^B} \right)^{(1)} \right\| + \beta_{1} + \left\|  \left(\overline{ \n_{j-i-1}^{A}} \right)^{(2)}  \right\|, \]
as well as
\[  \left\| \overline{\la_{n-j+1}^{(i,j-1)}}  | \overline{ \n_{j-i-1}^{A}} \right\| = \left\|  \left( \overline{\la_{n-j+1}^{(i,j-1)}}  \right)^{(1)} \right\| + \left\|  \left(\overline{ \n_{j-i-1}^{A}} \right)^{(2)} \right\|. \]
However, since these orders above are both \emph{minimum} among all matrices of the appropriate size, we must have
\[  \left\|  \left(\overline{\la_{n-j+1}^B} \right)^{(1)} \right\| =  \left\|  \left( \overline{\la_{n-j+1}^{(i,j-1)}}  \right)^{(1)} \right\| \]
since, if one had order strictly smaller than the other, we could replace one of the blocks above with a block of smaller order, contradicting the claim we reached the minimum. Consequently, we have proved our claim, and conclude:
\[ \left\| \overline{\la_{n-j+1}^{(i,j-1)}}  | \overline{ \n_{j-i-1}^{A}} \right\| = \left\| \overline{\la_{n-j+1}^B}  | \overline{ \n_{j-i-1}^{A}} \right\| .  \]
Thus, it remains to prove that

 \begin{equation}
 \left\| \la_{n-j+1}^{B}  \oplus  \n_{j-i}^{A} \right\|
 - \left\| \la_{n-j+1}^{B}  \oplus  \n_{j-i-1}^{A} \right\|
 \leq
 \left\| \la_{n-j}^{B}  \oplus  \n_{j-i+1}^{C} \right\|
  -\left\| \la_{n-j}^{B}  \oplus  \n_{j-i}^{C} \right\| \label{final} \end{equation}
  subject to the conditions of Equations~\ref{condition A},~\ref{condition B}, and~\ref{condition C}. Arguing matricially again allows us to assume the above inequality may be expressed as:

 \begin{multline*}  \left\| \begin{array}{c|c} \left(\overline{ \la_{n-j+1}^{B}}\right)^{(1)} &  \left(\overline{ \n_{j-i}^{A}} \right)^{(1)}  \\ \hline 0 & \left( \overline{\n_{j-i}^{A}}\right)^{(2)} \\ \hline
 0 & 0 \end{array} \right\|
 -  \left\| \begin{array}{c|c} \left(\overline{ \la_{n-j+1}^{B}}\right)^{(1)} &  \left(\overline{ \n_{j-i-1}^{A}} \right)^{(1)}  \\ \hline 0 & \left( \overline{\n_{j-i-1}^{A}}\right)^{(2)} \\ \hline
 0 & 0 \end{array} \right\|  \leq \\
  \left\| \begin{array}{c|c}  \left(\overline{ \la_{n-j}^{B}}\right)^{(1)} &  \left( \overline{\n_{j-i+1}^{C}}\right)^{(1)}  \\ \hline 0 & \left( \overline{\n_{j-i+1}^{C}}\right)^{(2)} \\ \hline
 0 &  \left( \overline{\n_{j-i+1}^{C}}\right)^{(3)} \end{array} \right\|
 -  \left\| \begin{array}{c|c}   \left(\overline{ \la_{n-j}^{B}}\right)^{(1)} & \left( \overline{\n_{j-i}^{C}}\right)^{(1)} \\ \hline 0 & \left( \overline{\n_{j-i}^{C}}\right)^{(2)} \\ \hline
 0 &  \left( \overline{\n_{j-i}^{C}}\right)^{(3)}\end{array} \right\|, \end{multline*}

By row operations (applied across all four matrix pairs simultaneously) and column operations on the left-hand blocks we may assume that \emph{both} the blocks  $\left(\overline{ \la_{n-j+1}^{B}}\right)^{(1)} $ and $ \left(\overline{ \la_{n-j}^{B}}\right)^{(1)}$ are diagonal (and, in particular, $ \left(\overline{ \la_{n-j}^{B}}\right)^{(1)}$ is the upper left corner of $\left(\overline{ \la_{n-j+1}^{B}}\right)^{(1)} $). This, in turn, implies that the bottom row of the block  $\left( \overline{\n_{j-i}^{A}}\right)^{(2)}$ is at the same height as the bottom row of $ \left( \overline{\n_{j-i+1}^{C}}\right)^{(2)}$.

By row operations below row $n-j+1$, and column operations on the right-hand block, we may assume $ \left( \overline{\n_{j-i}^{A}}\right)^{(2)} $ is a diagonal matrix whose invariant factors are decreasing down the diagonal.

 The matrix order
 \[ \left\| \begin{array}{c|c}  \left(\overline{ \la_{n-j}^{B}}\right)^{(1)} &  \left( \overline{\n_{j-i+1}^{C}}\right)^{(1)}  \\ \hline 0 & \left( \overline{\n_{j-i+1}^{C}}\right)^{(2)} \\ \hline
 0 &  \left( \overline{\n_{j-i+1}^{C}}\right)^{(3)} \end{array} \right\| \]
 is the sum of the order of the upper left block $ \left(\overline{ \la_{n-j}^{B}}\right)^{(1)}$
 and the minimal order among all $(j-i+1) \times (j-i+1)$ minors formed within the combined rows of the blocks $\left(\overline{\n_{j-i+1}^{C}}\right)^{(2)} $ and $\left(\overline{\n_{j-i+1}^{C}}\right)^{(3)}$. That is, we must compute the matrix order
 \[ \left\| \begin{array}{c}  \left( \overline{\n_{j-i+1}^{C}}\right)^{(2)} \\ \hline
  \left( \overline{\n_{j-i+1}^{C}}\right)^{(3)} \end{array} \right\| . \]
  To accomplish this, we may actually perform row operations (on all four matrix pairs simultaneously) by adding multiples of the rows in $ \left( \overline{\n_{j-i+1}^{C}}\right)^{(3)}$ to \emph{higher} rows of the block  $\left( \overline{\n_{j-i+1}^{C}}\right)^{(2)}$, since such operations will preserve the block decomposition of the matrix pair
\[ \left[ \begin{array}{c|c} \left(\overline{ \la_{n-j+1}^{B}}\right)^{(1)} &  \left(\overline{ \n_{j-i}^{A}} \right)^{(1)}  \\ \hline 0 & \left( \overline{\n_{j-i}^{A}}\right)^{(2)} \\ \hline
 0 & 0 \end{array} \right] \]
 and also the diagonal invariants in the block $ \left( \overline{\n_{j-i}^{A}}\right)^{(2)}$. Thus, we may ensure that the determinant of minimal order is the block $ \left( \overline{\n_{j-i+1}^{C}}\right)^{(2)}$. Further, we may then, by column operations, ensure that this block, too, is in diagonal form. However, we cannot control the arrangement of the invariant factors. From these constructions, we may conclude that the matrix order inequality
 \begin{multline*}  \left\| \begin{array}{c|c} \left(\overline{ \la_{n-j+1}^{B}}\right)^{(1)} &  \left(\overline{ \n_{j-i}^{A}} \right)^{(1)}  \\ \hline 0 & \left( \overline{\n_{j-i}^{A}}\right)^{(2)} \\ \hline
 0 & 0 \end{array} \right\|
 -  \left\| \begin{array}{c|c} \left(\overline{ \la_{n-j+1}^{B}}\right)^{(1)} &  \left(\overline{ \n_{j-i-1}^{A}} \right)^{(1)}  \\ \hline 0 & \left( \overline{\n_{j-i-1}^{A}}\right)^{(2)} \\ \hline
 0 & 0 \end{array} \right\|  \leq \\
  \left\| \begin{array}{c|c}  \left(\overline{ \la_{n-j}^{B}}\right)^{(1)} &  \left( \overline{\n_{j-i+1}^{C}}\right)^{(1)}  \\ \hline 0 & \left( \overline{\n_{j-i+1}^{C}}\right)^{(2)} \\ \hline
 0 &  \left( \overline{\n_{j-i+1}^{C}}\right)^{(3)} \end{array} \right\|
 -  \left\| \begin{array}{c|c}   \left(\overline{ \la_{n-j}^{B}}\right)^{(1)} & \left( \overline{\n_{j-i}^{C}}\right)^{(1)} \\ \hline 0 & \left( \overline{\n_{j-i}^{C}}\right)^{(2)} \\ \hline
 0 &  \left( \overline{\n_{j-i}^{C}}\right)^{(3)}\end{array} \right\| \end{multline*}
 reduces to proving

 \begin{equation} \left\|  \left( \overline{\n_{j-i}^{A}}\right)^{(2)} \right\| - \left\|  \left( \overline{\n_{j-i-1}^{A}}\right)^{(2)} \right\| \leq
 \left\| \left( \overline{\n_{j-i+1}^{C}}\right)^{(2)} \right\| - \left\|  \left( \overline{\n_{j-i}^{C}}\right)^{(2)} \right\|. \label{last} \end{equation}
 Note that in this case, by Equations~\ref{condition A}, the columns of $ \left( \overline{\n_{j-i-1}^{A}}\right)^{(2)} $ are in the span of the columns of $\left( \overline{\n_{j-i}^{A}}\right)^{(2)}$ (and similarly the columns of $\left( \overline{\n_{j-i}^{C}}\right)^{(2)}$ are in the span of $ \left( \overline{\n_{j-i+1}^{C}}\right)^{(2)}$ by Equation~\ref{condition C}). Since the matrix
  \[  \left\| \begin{array}{c|c} \left(\overline{ \la_{n-j+1}^{B}}\right)^{(1)} &  \left(\overline{ \n_{j-i-1}^{A}} \right)^{(1)}  \\ \hline 0 & \left( \overline{\n_{j-i-1}^{A}}\right)^{(2)} \\ \hline
 0 & 0 \end{array} \right\|  \]
 must be such that the resulting order is \emph{minimal}, the block $\left( \overline{\n_{j-i-1}^{A}}\right)^{(2)}$ must have minimal order among all rank $(j-i-1)$ submodules in the span of $\left( \overline{\n_{j-i}^{A}}\right)^{(2)}$. This implies
 \begin{equation}  \left( \overline{\n_{j-i-1}^{A}}\right)^{(2)}  = \left( \overline{\n_{j-i}^{A}}\right)^{(2)}\!|^{2}_{j-i}, \label{block A} \end{equation}
 and by the same reasoning
 \begin{equation} \left( \overline{\n_{j-i}^{C}}\right)^{(2)} = \left( \overline{\n_{j-i+1}^{C}}\right)^{(2)}\!|^{2}_{j-i+1}. \label{block C} \end{equation}

 Thus, we will argue that Inequality~\ref{last} holds where the following conditions have been established:
 \begin{enumerate}
 \item The left-hand side is the largest invariant factor of $ \left( \overline{\n_{j-i}^{A}}\right)^{(2)} $ (by Equation~\ref{block A}).
 \item The right-hand side is the largest invariant factor of  $\left( \overline{\n_{j-i+1}^{C}}\right)^{(2)} $ (by Equation~\ref{block C}).
 \item The block $ \left( \overline{\n_{j-i}^{A}}\right)^{(2)} $ is diagonal, and the orders of the entries are the invariant factors of the block, arranged in decreasing order.
 \item The block $\left( \overline{\n_{j-i+1}^{C}}\right)^{(2)} $ is also diagonal, but the arrangement of the orders of the entries is not (yet) determined.
  \item The rows below the block $ \left( \overline{\n_{j-i}^{A}}\right)^{(2)} $ are all zero.
 \item The bottom rows of both the blocks $ \left( \overline{\n_{j-i}^{A}}\right)^{(2)} $ and  $\left( \overline{\n_{j-i+1}^{C}}\right)^{(2)} $ lie in the same row of the matrices from which they come.

 \end{enumerate}

 By the first two statements above, we will have proved the right-leaning rhombus inequality if we can (finally) conclude that the largest invariant factor of the block $ \left( \overline{\n_{j-i}^{A}}\right)^{(2)} $ is less than or equal to the largest invariant factor of the block $\left( \overline{\n_{j-i+1}^{C}}\right)^{(2)} $.

We claim that the orders of the diagonal entries of  $\left( \overline{\n_{j-i+1}^{C}}\right)^{(2)} $ must be the \emph{same} as those in the corresponding rows of $ \left( \overline{\n_{j-i}^{A}}\right)^{(2)} $. This follows since if an entry in some row $k$ of one of the blocks had a \emph{lower} order than and entry in the same row of the other block (starting from the right-most column and working left), we could replace the column of the larger order with that of the smaller, and in the matrix forms we have here obtained, the resulting order of the matrix pairs would necessarily \emph{decrease}, which would contradict that we had achieved the minimum possible order already.

Thus, $(j-i)$ out of the $(j-i+1)$ many invariant factors of $\left( \overline{\n_{j-i+1}^{C}}\right)^{(2)} $ are precisely the invariant factors of the block $ \left( \overline{\n_{j-i}^{A}}\right)^{(2)} $. From this we conclude that the largest invariant factor of $ \left( \overline{\n_{j-i}^{A}}\right)^{(2)} $ cannot exceed the largest invariant factor of $\left( \overline{\n_{j-i+1}^{C}}\right)^{(2)} $, and the inequality is proved.

\bigskip
\bigskip

\noindent \underline{Vertical Rhombus Inequality for the $\{ h_{ij} \}$.}

\medskip

Here we wish to prove our formula
\[ h_{st} =  \  | \lambda | - \min_{\substack{ \Lambda_{n-t} \oplus {\scriptstyle{\cal N }_{t-s}}}} \!\!\big(  \left\| \la_{n-t}+ \n_{t-s} \right\| \big) \]
satisfies:
\[ h_{ij} + h_{(i+1)j} \leq h_{(i+1)(j+1)} + h_{i(j-1)}. \]
Let us express Equation~\eqref{min formula}, written in terms of modules realizing the minima appearing in it:

\begin{multline}  \left\| \la_{n-j}^{(ij)}\oplus \n_{j-i}^{(ij)} \right\| - \left\| \la_{n-j-1}^{(i+1)(j+1)}\oplus \n_{j-i}^{(i+1)(j+1)} \right\| \leq  \\
  \left\| \la_{n-j+1}^{i(j-1)}\oplus \n_{j-i-1}^{i(j-1)} \right\| -  \left\| \la_{n-j}^{(i+1)j}\oplus \n_{j-i-1}^{(i+1)j} \right\| . \end{multline}

Let us, in fact, swap the order of the summands in the above:

\begin{multline}  \left\| \n_{j-i}^{(ij)}  \oplus  \la_{n-j}^{(ij)} \right\| - \left\|  \n_{j-i}^{(i+1)(j+1)}  \oplus\la_{n-j-1}^{(i+1)(j+1)}\right\| \leq  \\
  \left\| \n_{j-i-1}^{i(j-1)}  \oplus \la_{n-j+1}^{i(j-1)}\right\| -  \left\| \n_{j-i-1}^{(i+1)j} \oplus \la_{n-j}^{(i+1)j} \right\| . \end{multline}
Let us compare the above to Inequality~\ref{main} from the right-leaning case above:
\begin{multline}
 \left\| \la_{n-j+1}^{(i-1,j-1)}  \oplus  \n_{j-i}^{(i-1,j-1)} \right\|
 - \left\| \la_{n-j+1}^{(i,j-1)}  \oplus  \n_{j-i-1}^{(i,j-1)} \right\|
 \leq \\
 \left\| \la_{n-j}^{(i-1,j)}  \oplus  \n_{j-i+1}^{(i-1,j)} \right\|
  -\left\| \la_{n-j}^{(ij)}  \oplus  \n_{j-i}^{(ij)} \right\|  \end{multline}

 We see in both cases that the pattern of ranks (by abuse of notation) is:
 \[ \| s \oplus k \| - \| s \oplus (k-1) \| \leq \| (s-1) \oplus (k+1) \| - \| (s-1) \oplus k \|. \]
 As such, we may argue exactly as in the right-leaning case (which only depended on the relative sizes of these ranks), and conclude the vertical rhombus inequality holds as well.

\bigskip

\bigskip
 \noindent \underline{Left-Leaning Rhombus Inequality for the $\{ h_{ij} \}$.}

\medskip

Here we wish to prove our formula given by Equation~\eqref{min formula}:
\[ h_{st} =   \  | \lambda | - \min_{ \Lambda_{n-j} \oplus {\scriptstyle{\cal N}_{j-i}}} \!\!\big(  \left\| \la_{n-j}  +  \n_{j-i} \right\| \big) \]
satisfies:
\[ h_{ij} + h_{i(j-1)} \geq h_{(i-1)(j-1)} + h_{(i+1)j}. \]

Let us use Equation~\eqref{min formula}, written in terms of modules realizing the minima appearing in it, simplifying after subtracting the terms $| \lambda |$ appearing on both sides:

\begin{multline}
  \left\| \la_{n-j+1}^{(i,j-1)}  \oplus  \n_{j-i-1}^{(i,j-1)} \right\| -
   \left\| \la_{n-j}^{(i+1,j)}  \oplus  \n_{j-i-1}^{(i+1,j)} \right\|
 \leq \\
  \left\| \la_{n-j+1}^{(i-1,j-1)}  \oplus  \n_{j-i}^{(i-1,j-1)} \right\|-
 \left\| \la_{n-j}^{(ij)}  \oplus  \n_{j-i}^{(ij)} \right\|.
  \end{multline}

 To emphasize similarity to earlier cases, let us swap the order of the summands in the above:

\begin{multline}
  \left\|  \n_{j-i-1}^{(i,j-1)}  \oplus \la_{n-j+1}^{(i,j-1)}  \right\| -
   \left\|  \n_{j-i-1}^{(i+1,j)}  \oplus  \la_{n-j}^{(i+1,j)}  \right\|
 \leq \\
  \left\| \n_{j-i}^{(i-1,j-1)}  \oplus   \la_{n-j+1}^{(i-1,j-1)} \right\|-
 \left\|  \n_{j-i}^{(ij)}  \oplus  \la_{n-j}^{(ij)} \right\|.
  \end{multline}
  At this point we can repeat the initial constructions as in the right-leaning case, replacing the submodules above with certain aligned submodules:

\[
  \left\|  \n_{j-i-1}^{(B)}  \oplus \la_{n-j+1}^{(A)}  \right\| -
   \left\|  \n_{j-i-1}^{(B)}  \oplus  \la_{n-j}^{(A)}  \right\|
 \leq
  \left\| \n_{j-i}^{(B)}  \oplus   \la_{n-j+1}^{(C)} \right\|-
 \left\|  \n_{j-i}^{(B)}  \oplus  \la_{n-j}^{(C)} \right\|.
\]
where, if two submodules have the same letter superscript, either they equal (if their ranks are the same), or the ranks differ by $1$, in which case the submodule of smaller rank is spanned by an invariant adapted basis of the submodule of larger rank, where the generators correspond to all but the largest invariant factor. Thus, as in the right-leaning case, we may write the above matricially (after suitable row and column operations):

 \begin{multline*}  \left\| \begin{array}{c|c}\left( \overline{\n_{j-i-1}^{B}}\right)^{(1)}  &  \left(\overline{ \n_{j-i+1}^{A}} \right)^{(1)}  \\ \hline 0 & \left(\overline{ \la_{n-j+1}^{A}}\right)^{(2)} \\ \hline
 0 & 0 \end{array} \right\|
 -  \left\| \begin{array}{c|c} \left( \overline{\n_{j-i-1}^{B}}\right)^{(1)}&  \left(\overline{ \n_{j-i}^{A}} \right)^{(1)}  \\ \hline 0 &  \left(\overline{ \la_{n-j}^{A}}\right)^{(2)} \\ \hline
 0 & 0 \end{array} \right\|  \leq \\
  \left\| \begin{array}{c|c} \left( \overline{\n_{j-i}^{B}}\right)^{(1)}  &  \left(\overline{ \la_{n-j+1}^{C}}\right)^{(1)}  \\ \hline 0 & \left(\overline{ \la_{n-j+1}^{C}}\right)^{(2)} \\ \hline
 0 & \left(\overline{ \la_{n-j+1}^{C}}\right)^{(3)} \end{array} \right\|
 -  \left\| \begin{array}{c|c}   \left( \overline{\n_{j-i}^{B}}\right)^{(1)}  &\left(\overline{ \la_{n-j}^{C}}\right)^{(1)} \\ \hline 0 &\left(\overline{ \la_{n-j}^{C}}\right)^{(2)}  \\ \hline
 0 &  \left(\overline{ \la_{n-j}^{C}}\right)^{(3)} \end{array} \right\|, \end{multline*}
We may repeat the constructions of the right-leaning rhombus inequality to the left-leaning case here, and may assume:
 \begin{enumerate}
 \item The left-hand side is the largest invariant factor of  $\left(\overline{ \la_{n-j+1}^{A}}\right)^{(2)}  $.
 \item The right-hand side is the largest invariant factor of  $ \left(\overline{ \la_{n-j+1}^{C}}\right)^{(2)}  $.
 \item The block $\left(\overline{ \la_{n-j+1}^{A}}\right)^{(2)}  $ is diagonal, and the orders of the entries are the invariant factors of the block, arranged in \emph{increasing} order down the diagonal (this is unlike the case for the proof of the right-leaning rhombus inequality).
 \item The block $ \left(\overline{ \la_{n-j+1}^{C}}\right)^{(2)}  $ is also diagonal, but the arrangement of the orders of the entries is not (yet) determined.
  \item The rows below the block $ \left( \overline{\n_{j-i}^{A}}\right)^{(2)} $ are all zero.
 \end{enumerate}
 However, the sixth condition in the right-leaning rhombus inequality fails. Namely, the blocks $\left(\overline{ \la_{n-j+1}^{A}}\right)^{(2)}  $ and  $ \left(\overline{ \la_{n-j+1}^{C}}\right)^{(2)}  $ do \emph{not} lie along the same row (and, in this case, both are of the same rank). However, what we do see is that in the proposed matricial inequality:

  \begin{multline*}  \left\| \begin{array}{c|c}\left( \overline{\n_{j-i-1}^{B}}\right)^{(1)}  &  \left(\overline{ \n_{j-i+1}^{A}} \right)^{(1)}  \\ \hline 0 & \left(\overline{ \la_{n-j+1}^{A}}\right)^{(2)} \\ \hline
 0 & 0 \end{array} \right\|
 -  \left\| \begin{array}{c|c} \left( \overline{\n_{j-i-1}^{B}}\right)^{(1)}&  \left(\overline{ \n_{j-i}^{A}} \right)^{(1)}  \\ \hline 0 &  \left(\overline{ \la_{n-j}^{A}}\right)^{(2)} \\ \hline
 0 & 0 \end{array} \right\|  \leq \\
  \left\| \begin{array}{c|c} \left( \overline{\n_{j-i}^{B}}\right)^{(1)}  &  \left(\overline{ \la_{n-j+1}^{C}}\right)^{(1)}  \\ \hline 0 & \left(\overline{ \la_{n-j+1}^{C}}\right)^{(2)} \\ \hline
 0 & \left(\overline{ \la_{n-j+1}^{C}}\right)^{(3)} \end{array} \right\|
 -  \left\| \begin{array}{c|c}   \left( \overline{\n_{j-i}^{B}}\right)^{(1)}  &\left(\overline{ \la_{n-j}^{C}}\right)^{(1)} \\ \hline 0 &\left(\overline{ \la_{n-j}^{C}}\right)^{(2)}  \\ \hline
 0 &  \left(\overline{ \la_{n-j}^{C}}\right)^{(3)} \end{array} \right\|, \end{multline*}
 The top of the $(n-j+1) \times (n-j+1)$ block $ \left(\overline{ \la_{n-j+1}^{A}}\right)^{(2)}$ lies \emph{one row higher} in its corresponding matrix than the block $ \left(\overline{ \la_{n-j+1}^{C}}\right)^{(2)}$ (of the same size) lies in its matrix. That is, the pair of blocks (arranged so corresponding rows are at the same height) would appear as:

  \[ \begin{array}{ccc}
  \begin{array}{c}
  \begin{array}{|cccc|} \hline
  t^{\alpha_{n-j+1}} &&& \\ \hline
   & t^{\alpha_{n-j}} && \\
   && \ddots & \\
   &&&t^{\alpha_{1}} \end{array} \\ \hline
   \ \ \ \ \  \\ \\
    \left(\overline{ \la_{n-j+1}^{A}}\right)^{(2)}  \end{array} & \ \ \ \ \ &
   \begin{array}{c}
   \ \ \ \ \ \\ \hline
   \begin{array}{|cccc|}
   t^{\beta_{n-j+1}} &&& \\
   & \ddots && \\
   && t^{\beta_{2}} & \\ \hline
   &&& t^{\beta_{1}} \end{array}\\ \hline \\
    \left(\overline{ \la_{n-j+1}^{C}}\right)^{(2)} \end{array}
   \end{array} \]

   \medskip

  Recall that the orders of the invariant factors $t^{\beta_{n-j+1}}, \ldots , t^{\beta_{1}}$ are not assumed to be in any particular order. Then, arguing as we did for the right-leaning rhombus inequality, we see that in fact, we must have $\alpha_{k} = \beta_{k+1}$ for $k = 1 \dots (n-j)$. In particular, the largest invariant factor of $  \left(\overline{ \la_{n-j+1}^{A}}\right)^{(2)}$, namely $\alpha_{1}$, cannot exceed the largest invariant factor of $\left(\overline{ \la_{n-j+1}^{C}}\right)^{(2)} $ since it must actually be among the invariant factors of $\left(\overline{ \la_{n-j+1}^{C}}\right)^{(2)} $.

  \medskip

  With this, the last inequality has been verified, and the proof is complete.
 \end{proof}

\section{Computing Types for Hives}
By Theorem~\ref{rhombus} we see that the values for $\{h_{st} \}$ determined by the formula in Equation~\ref{min formula} do determine a hive. What remains left to prove for Theorem~\ref{full theorem} is that the hive we have produced has the correct type, and also to prove the alternate formulas given by maxima of orders of various blocks.

\begin{lem}
Let $\la$ and $\la$ two full $\oi$-lattices in $\Gr$, and let ${\cal M} \in \Gr$ be determined by the condition that $(I,{\cal M})$ and $(\n, \la)$ are in the same $GL_{n}(\K)$ orbit. Fix a a common matrix identification of $\la, \n$ and ${\cal M}$. Let $U_{t}$ denote an arbitrary $\oi$-submodule of $K^n$ of rank $t$.

   Then
\[ \| \overline{\la(U_{t}) }\| = \| \overline{\n(U_t) }\| + \|\overline{ \m(U_t)} \|. \] \label{lam=mu+nu}
\end{lem}

\begin{proof}If $t = n$, the result is immediate by determinants. More generally,
let us choose a basis $\{ \vec{u}_{1}, \ldots, \vec{u}_{t} \}$ for $U_t$, and suppose $inv(\n(U_{t}))= (\alpha_{1}, \ldots , \alpha_{t})$. Then
\begin{align} [\overline{\la(U_{t})} ] &= [ \overline{\n \m(\vec{u}_{1}), \ldots , \n \m(\vec{u}_{t}) }] \\
& =\overline{ \n}[ \overline{\m(\vec{u}_{1}), \ldots , \m(\vec{u}_{t}) }]. \end{align}
We may assume in the above that $\{ \vec{u}_{1}, \ldots , \vec{u}_{t} \}$ is actually an invariant adapted basis for the submodule whose columns form $\overline{\m(U_{t})}$, so that
\begin{align} [ \overline{\m(\vec{u}_{1}), \ldots , \m(\vec{u}_{t}) }] & = [ \overline{t^{\alpha_{1}}\vec{u}_{1}, \ldots , t^{\alpha_{t}}\vec{u}_{t} }] \\
& = [ \overline{\vec{u}_{1}, \ldots , \vec{u}_{t} }] \cdot diag(t^{\alpha_{1}}, \ldots , t^{\alpha_t}). \end{align}
Consequently,
\begin{align*}
\| \overline{\la(U_{t}) }\|  & = \|  [ \overline{\n \m (\vec{u}_{1}), \ldots , \n \m (\vec{u}_{t}) }]  \| \\
& = \|\overline{ \n} [ \overline{\m (\vec{u}_{1}), \ldots , \m (\vec{u}_{t}) }] \| \\
& = \|\overline{ \n} [ \overline{\vec{u}_{1}, \ldots , \vec{u}_{t} }] \cdot diag(t^{\alpha_{1}}, \ldots , t^{\alpha_t}) \| \\
& = \|[ \overline{\n \vec{u}_{1}, \ldots , \n \vec{u}_{t} }] \| + \|\overline{ \m(U_{t})} \| \\
& = \| \overline{\n(U_t) }\| + \| \overline{\m(U_t)} \|. \end{align*}

\end{proof}

\begin{lem}Let $\Lambda$ and $N$ be two full $\oi$-lattices of $K^n$. Let $\Phi: K^n \rightarrow K^n$ be defined so that $\Phi(N) = \Lambda$, and let $\mu$ denote the invariant partition of ${\cal M} = \Phi(\oi^n)$. Let the invariant partition of $\la$ be $inv(\la) = (\lambda_{1} \geq \cdots \geq \lambda_{n})$, and then let $| \lambda | = \lambda_{1} + \cdots + \lambda_{n}$. Below, let $U_{s}$ denote a $\oi$-submodule of $K^n$ of rank $s$, and $V_{t}$ denote a $\oi$-submodule of rank $t$, etc..

   Then
 \begin{equation}
   | \lambda | - \ \min_{\substack{ \Lambda_{n-t} \oplus {\scriptstyle {\cal N}_{t-s}}}} \!\!\big(  \left\| \Lambda_{n-t}\oplus \n_{t-s} \right\| \big)
  = \max_{\substack{ \Lambda_{s} \oplus {\scriptstyle {\cal M}_{t-s}}}} \  \big(  \left\| \Lambda_{s} \oplus \m_{t-s} \right\| \big) \label{first eq}
     \end{equation}
\noindent and
\begin{equation}
  | \lambda | - \min_{\substack{ \Lambda_{n-t} \oplus {\scriptstyle {\cal M}_{t-s}}}} \!\!\big(  \left\| \Lambda_{n-t}\oplus \m_{t-s} \right\| \big) =\max_{\substack{ \Lambda_{s} \oplus {\scriptstyle {\cal N}_{t-s}}}} \  \big(  \left\| \Lambda_{s} \oplus \n_{t-s} \right\| \big)\label{2nd eq}
   % =& \left\| \Lambda\!\!\downarrow_{s}^{1} + N\!\downarrow_{t-s}^{1} \right\|
  \end{equation} \label{min=max}
  \end{lem}
\begin{proof} We shall choose a matrix identification of the modules $\la$ and $\n$ such that any submodules, $\la_{t}$ and $\n_{t-s}$, for instance, may be realized by means of submodules $U_{n-t}, V_{t-s} \subseteq \oi^n$ such that $\overline{\la_{t}} = \overline{\la(U_{n-t})}$ and $\overline{\n_{t-s}} = \overline{\n(V_{t-s})}$.
Then
\[   | \lambda | - \min_{ \Lambda_{n-t} \oplus {\scriptstyle{\cal N}_{t-s}}} \!\!\big(  \left\| \Lambda_{n-t}\oplus \n_{t-s} \right\| \big)
=  |\lambda | -\min_{ U_{n-t} \oplus V_{t-s}} \!\!\left(\left\| \overline{\Lambda(U_{n-t})}| \overline{\n(V_{t-s}) } \right\| \right). \]

 For now, fix some choice of $U_{n-t}$ and $V_{t-s}$, and let us also choose a complementary rank $s$ submodule $Y_{s}$ so that $\oi^n = V_{t-s} \oplus U_{n-t} \oplus Y_{s}$. We consider the matrix $\la = \la(\oi^n)$ (under our matrix identification), written with respect to this decomposition:
  \[ [\overline{\la(U_{n-t})} | \overline{\la(V_{t-s})}|\overline{\la(Y_{s})}]  .
  \]
  By means of row operations, and column operations that preserve the splitting by direct sums above, we may express this matrix in the block form:
  \[  [\overline{\la(U_{n-t})} | \overline{\la(V_{t-s})}|\overline{\la(Y_{s})}]  = \left[ \begin{array}{c|c|c} \left(   \overline{\la(V_{t-s}) }\right)^{(1)} &  \left( \overline{\la(U_{n-t}) }\right)^{(1)}& \left( \overline{ \la(Y_{s})} \right)^{(1)} \\
  0 &  \left( \overline{\la(U_{n-t})} \right)^{(2)} & \left( \overline{ \la(Y_{s})} \right)^{(2)} \\
  0 & 0 & \left( \overline{ \la(Y_{s})} \right)^{(3)} \end{array} \right]. \]

\medskip
\noindent {\bf Claim:} If the submodules $U_{n-t}$ and $V_{t-s}$ are chosen such that
\[ \left\| \overline{\Lambda(U_{n-t})} | \overline{\n(V_{t-s}) } \right\| \]
is minimal, then after choosing the complementary submodule $Y_{s}$, we may assume the block decomposition
\begin{equation}    [\overline{\la(U_{n-t})} | \overline{\la(V_{t-s})}|\overline{\la(Y_{s})}]  = \left[ \begin{array}{c|c|c} \left(   \overline{\la(V_{t-s}) }\right)^{(1)} &  \left( \overline{\la(U_{n-t}) }\right)^{(1)}& \left( \overline{ \la(Y_{s})} \right)^{(1)} \\
  0 &  \left( \overline{\la(U_{n-t})} \right)^{(2)} & \left( \overline{ \la(Y_{s})} \right)^{(2)} \\
  0 & 0 & \left( \overline{ \la(Y_{s})} \right)^{(3)} \end{array} \right] \label{block}
  \end{equation}
actually has the form
\[  [\overline{\la(U_{n-t})} | \overline{\la(V_{t-s})}|\overline{\la(Y_{s})}]  = \left[ \begin{array}{c|c|c} \left(   \overline{\la(V_{t-s}) }\right)^{(1)} &  \left( \overline{\la(U_{n-t}) }\right)^{(1)}& \left( \overline{ \la(Y_{s})} \right)^{(1)} \\
  0 &  \left( \overline{\la(U_{n-t})} \right)^{(2)} &0 \\
  0 & 0 & \left( \overline{ \la(Y_{s})} \right)^{(3)} \end{array} \right]. \]

  That is, we may assume the $(n-t) \times s$ block $ \left(\overline{ \la(Y_{s})} \right)^{(2)}$ is a matrix of zeros.

  \medskip

  \noindent \emph{Proof of Claim:} Let us first, by Lemma~\ref{normal} assume that the matrix realization of Equation~\ref{block} above is in normal form. In particular, the blocks
$ \left(\la(U_{n-t}) \right)^{(2)}$ and $ \left( \la(Y_{s}) \right)^{(3)}$ are diagonal. We consider column operations that add to columns of $ \left( \la(Y_{s}) \right)^{(2)}$ a multiple of columns of $ \left(\la(U_{n-t}) \right)^{(2)}$, and also row operations that add to a row of $ \left( \la(Y_{s}) \right)^{(2)}$ a multiple of a row of $ \left( \la(Y_{s}) \right)^{(3)}$. We use such operations to ensure that any non-zero entry in $ \left( \la(Y_{s}) \right)^{(2)}$ (if it exists), must have order strictly \emph{less} than the sole non-zero entry to its left in $ \left(\la(U_{n-t}) \right)^{(2)}$, or sole non-zero element below it in $ \left( \la(Y_{s}) \right)^{(3)}$.

We argue that after this process, assuming $ \left\| \overline{\Lambda(U_{n-t})} | \overline{\n(V_{t-s}) } \right\|$ is of minimal possible order among appropriate submodules and ranks, that the block $ \left( \overline{ \la(Y_{s}) }\right)^{(2)}$ is all zeros. Suppose this was not the case, and there is some element$ \gamma_{ij} \neq 0$ appearing in row $i$ of  $ \left( \overline{ \la(U_{n-t})} \right)^{(2)}$ and column $j$ of $ \left(  \overline{ \la(Y_{s})} \right)^{(3)}$. We may assume that all non-zero elements of $ \left( \overline{  \la(Y_{s}) }\right)^{(2)}$ lying strictly below $ \gamma_{ij} $ have order at least $\| \gamma_{ij} \|$. As noted above, we must also have $\| \gamma_{ij} \|$ is \emph{strictly} less than both the order of the element in row $i$ of the diagonal matrix $\left( \overline{ \la(U_{n-t})} \right)^{(2)}$, and also the order of the element in column $j$ of $ \left(  \overline{ \la(Y_{s})} \right)^{(3)}$:

\[ \begin{array}{cc}
 \overline{ \la (U_{n-t})}^{(2)} &  \overline{ \la(V_{s})}^{(2)} \\
\left[ \begin{array}{ccccc}
t^{\beta_{1}} & 0 & \dots & \dots & 0\\
0 & \ddots & \ddots  & & \vdots \\
\vdots & & t^{\beta_{i}} & &\vdots \\
\vdots & & & \ddots & 0 \\
0 & \dots & \dots & 0 & t^{\beta_{n-t}} \end{array} \right] &
\left[ \begin{array}{ccccc}
\ \ \ \ \   \ \ \  &\ & \vdots & \  &\ \ \ \ \ \ \  \ \\
\  & \  & \vdots  & \ &\  \\
 & & \gamma_{ij} & & \\
 & & \vdots  & &  \\
&  & \vdots  & &  \end{array} \right]  \\
\left[ \begin{array}{ccccc}
\  0 \ \ \ & 0 & \dots & \dots & \ \ \  0\  \\
0 & \ddots & \ddots  & & \vdots \\
\vdots & & 0 & &\vdots \\
\vdots & & & \ddots & 0 \\
0 & \dots & \dots & 0 & 0 \end{array} \right] &
\left[ \begin{array}{ccccc}
t^{\alpha_{1}} & 0 & \dots & \dots & 0\\
0 & \ddots & \ddots  & & \vdots \\
\vdots & & t^{\alpha_{j}} & &\vdots \\
\vdots & & & \ddots & 0 \\
0 & \dots & \dots & 0 & t^{\alpha_{s}}
\end{array} \right] \\
 \overline{ \la (U_{n-t})}^{(3)} &  \overline{ \la (Y_{s})}^{(3)}
\end{array} \]

Suppose we now \emph{swap} column $j$ (containing the entry $\gamma_{ij}$ in row $i$), with the column of $ \overline{ \la(U_{n-t}) }$ containing the sole non-zero entry in row $i$ of $ \left( \overline{ \la(U_{n-t}) }\right)^{(2)}$.  Since the orders below $\gamma_{ij}$ are all at least of order $\| \gamma_{ij} \|$, we may use row operations to put this new version of $\left( \overline{ \la(U_{n-t})} \right)^{(2)}$ into upper triangular form. Call this new block $\left( \overline{ \la(U_{n-t})} \right)^{(2)\star}$.:

\[
 \begin{array}{c}
 (``\gamma_{ij}") \rightarrow  \overline{ \la (U_{n-t})}^{(2)}  \\
\left[ \begin{array}{ccccc}
t^{\beta_{1}} & 0 &\vdots & \dots & 0\\
0 & \ddots &\vdots   & & \vdots \\
\vdots & & \gamma_{ij} & &\vdots \\
\vdots & & \vdots & \ddots & 0 \\
0 & \dots & \vdots &\dots & t^{\beta_{n-t}} \end{array} \right]   \\
\left[ \begin{array}{ccccc}
\  0 \ \ \ & 0 & 0 & \dots & \ \ \  0\  \\
0 & \ddots & \vdots  & & \vdots \\
\vdots & & t^{\alpha_{j}}& &\vdots \\
\vdots & & & \ddots & 0 \\
0 & \dots & 0 &\dots  & 0 \end{array} \right]
\end{array} \Rightarrow
 \begin{array}{c}
 \overline{ \la (U_{n-t})}^{(2)\star} \\
\left[ \begin{array}{ccccc}
t^{\beta_{1}} & 0 &\vdots & \dots & 0\\
0 & \ddots &\vdots   & & \vdots \\
\vdots & & \gamma_{ij} & &\vdots \\
\vdots & &0 & \ddots & 0 \\
0 & \dots & \vdots &\dots & t^{\beta_{n-t}} \end{array} \right]   \\
\left[ \begin{array}{ccccc}
 \ 0\   \ \  & 0 & 0 & \dots & \ \  \  0\ \ \   \\
0 & \ddots & \vdots  & & \vdots \\
\vdots & & 0& &\vdots \\
\vdots & & & \ddots & 0 \\
0 & \dots & 0 &\dots  & 0 \end{array} \right]
\end{array}
\]

 so that the entry in row $i$ is strictly lower than before, resulting in
\[  \left\|  \overline{  \Lambda(U_{n-t})}^{\star}|  \overline{ \n(V_{t-s})}  \right\|  <  \left\|  \overline{ \Lambda(U_{n-t})} |  \overline{ \n(V_{t-s}) } \right\| , \]
(where $ \overline{ \Lambda(U_{n-t})}^{\star}$ is the matrix realization with the swapped column), contradicting minimality and proving the claim.

Consider also the block

\[ \left[ \begin{array}{c} \left(   \overline{ \la(V_{t-s})} \right)^{(1)} \\ 0 \\ 0 \end{array} \right] = \overline{  \n \m( V_{t-s} )}. \]
Choose an invariant adapted basis  $\{ t^{\kappa_{1}}\vec{v}_{1}, \ldots ,  t^{\kappa_{s}}\vec{v}_{t-s} \} $ for $\m\left[ V_{t-s} \right]$ so that
\begin{align*}
 \left[ \begin{array}{c} \left(  \overline{ \la(V_{t-s})} \right)^{(1)} \\ 0 \\ 0 \end{array} \right] & = \overline{  \n \m( V_{t-s} )} \\
 & = \overline{\n} \left[\overline{ \m}\vec{v}_{1}, \ldots , \overline{\m} \vec{v}_{t-s} \right] \\
 & =\overline{ \n} \left[ \vec{v}_{1} , \ldots , \vec{v}_{t-s} \right] \cdot diag(t^{\kappa_{1}}, \ldots, t^{\kappa_{t-s}}) \end{align*}
 where $\kappa_{1}, \ldots , \kappa_{t-s}$ are the invariant factors of $\overline{\m (V_{t-s})}$. Let $\Delta =  diag(t^{\kappa_{1}}, \ldots, t^{\kappa_{t-s}}) $, so that, in particular,
 \[ \overline{ \n (V_{t-s}) }= \overline{\n} \left[ \vec{v}_{1} , \ldots , \vec{v}_{t-s} \right] = \left[ \begin{array}{c} \left(  \overline{\la(V_{t-s})} \right)^{(1)} \cdot \Delta^{-1} \\ 0 \\ 0 \end{array} \right]  . \]
 Thus we have
 \begin{align}
 \left\| \overline{\n (V_{t-s})} \right\| + \left\| \left(\overline{\la(U_{n-t})} \right)^{(2)} \right\|  & = \left\|  \begin{array}{c|c} \left(  \overline{\la(V_{t-s})} \right)^{(1)} \cdot \Delta^{-1} &\left(\overline{\la(U_{n-t})} \right)^{(1)}\\
 0 & \left(\overline{\la(U_{n-t})} \right)^{(2)}  \\
 0 & 0
 \end{array}
  \right\| \\
  &= \left\|\overline{ \n(V_{t-s})} | \overline{\la(U_{n-t})} \right\|. \label{nu-sance}
   \end{align}

 A similar argument establishes
 \begin{equation} \left\|\overline{ \m(V_{t-s})} \right\| + \left\|  \left( \overline{\la(Y_{s})} \right)^{(3)}  \right\| =  \left\|  \overline{\m(V_{t-s})} |  \overline{\la(Y_{s})} \right\| . \label{mu-sance} \end{equation}

We may then argue:
\begin{align*} | \lambda | & = \left\|  \begin{array}{c|c|c} \left(   \overline{\la(V_{t-s}) }\right)^{(1)} &  \left( \overline{\la(U_{n-t}) }\right)^{(1)}& \left( \overline{ \la(Y_{s})} \right)^{(1)} \\
  0 &  \left( \overline{\la(U_{n-t})} \right)^{(2)} &0 \\
  0 & 0 & \left( \overline{ \la(Y_{s})} \right)^{(3)} \end{array}  \right\| \\
  & = \left\|  \left(   \overline{\la(V_{t-s}) }\right)^{(1)} \right\| + \left\| \left( \overline{\la(U_{n-t})} \right)^{(2)}\right\| + \left\|  \left( \overline{ \la(Y_{s})} \right)^{(3)}  \right\| \\
  & = \left\| \overline{ \la(V_{t-s})} \right\| +\left\|  \left(  \overline{\la(U_{n-t})} \right)^{(2)} \right\| + \left\|  \left(  \overline{\la(Y_{s})} \right)^{(3)}  \right\| \\
  & = \left\|  \overline{\m(V_{t-s})} \right\| + \left\|  \overline{\n (V_{t-s})} \right\| + \left\| \left( \overline{ \la(U_{n-t})} \right)^{(2)} \right\| + \left\|  \left( \overline{ \la(Y_{s})} \right)^{(3)}  \right\|, \  \hbox{by Lemma~\ref{lam=mu+nu}}\\
  & =  \left\| \overline{\n (V_{t-s})}\right\| + \left\| \left(\overline{\la(U_{n-t})} \right)^{(2)} \right\| + \left\| \overline{\m(V_{t-s})} \right\| + \left\|  \left(\overline{ \la(Y_{s})} \right)^{(3)}  \right\|  \\
  & = \left\| \overline{\n(V_{t-s})} | \overline{\la(U_{n-t})} \right\| + \left\| \overline{  \m(V_{t-s})} |  \overline{\la(Y_{s})} \right\|, \ \ \hbox{by Equations~\ref{nu-sance} and~\ref{mu-sance}.} \end{align*}

Therefore, since
\[ | \lambda | =  \| \overline{ \la(U_{n-t}) }| \overline{ \n(V_{t-s})} \| + \| \overline{\m(V_{t-s})} | \overline{\la(Y_{s})}\|, \]
we must have the minimal value attained by any expression of the form
\[   \| \overline{ \la(U_{n-t})}  | \overline{\n(V_{t-s})} \|  \]
equals the maximal value of a corresponding expression
\[ \| \overline{\m(V_{t-s})}| \overline{\la(Y_{s})} \|. \]
From this Equation~\eqref{first eq} is proved. Equation~\eqref{2nd eq} is proved analogously.
\end{proof}

\begin{thm}[Theorem~\ref{main theorem} above]
    Let $\Lambda$ and $N$ two full $\oi$-lattices of $K^n$. Let the invariant partition of $N$ be $inv(N) = \nu = (\nu_{1} , \ldots , \nu_{n})$ and $inv(\Lambda) = \lambda = (\lambda_{1} , \ldots , \lambda_{n})$. Let $\Phi: K^n \rightarrow K^n$ be defined so that $\Phi(N) = \Lambda$, and then let $\mu$ denote the invariant partition of ${\cal M} = \Phi(\oi^n)$. Let $| \lambda | = \lambda_{1} + \cdots + \lambda_{n}$. Below, let $U_{s}$ denote a $\oi$-submodule of $K^n$ of rank $s$, and $V_{t}$ denote a $\oi$-submodule of rank $t$, etc..

   Then setting
\begin{align}   h_{st} &=  \  | \lambda | - \min_{ \Lambda_{n-t} \oplus {\scriptstyle {\cal N}_{t-s}}} \!\!\big(  \left\| \Lambda_{n-t}+ N_{t-s} \right\| \big)\label{h min} \\
   & = \qquad \max_{\Lambda_{s} \oplus  {\cal M}_{t-s}} \!\!\left( \left\|  \Lambda_{s} \oplus  {\cal M}_{t-s} \right\| \right) \label{h max}
     \end{align}
     forms a hive of type $(\mu,\nu,\lambda)$, and

\begin{align}   h_{st'} &=  \  | \lambda | - \min_{\substack{ \Lambda_{n-t} \oplus {\cal M}_{t-s}}} \!\!\big(  \left\| \Lambda_{n-t}+ {\cal M}_{t-s} \right\| \big) \label{h' min}\\
   & = \qquad \max_{\Lambda_{s} \oplus  {\scriptstyle {\cal N}_{t-s}}} \!\!\left( \left\|  \Lambda_{s} \oplus  \n_{t-s} \right\| \right) \label{h' max}
     \end{align}
  forms a hive of type $(\nu,\mu,\lambda)$.

\end{thm}

\begin{proof} By Theorem~\ref{min formula} we see that Equations~\eqref{h min} and~\eqref{h' min} do indeed define hives, and by Lemma~\ref{min=max} we may express this hive in the alternate form of Equations~\eqref{h max} and~\eqref{h' max}. What is left is to show that these hives have the types given above.

\medskip

To prove the set $\{ h_{ij} \}$ given by Equations~\eqref{h min} and~\eqref{h max} determines a hive of type $(\mu, \nu,\lambda)$, we start along the left edge.
Along the left edge of the hive all entries have the form $h_{0k}$. Usinq Equation~\eqref{h max} we see
\begin{align*} h_{0k} & = \max_{\Lambda_{0} \oplus  {\cal M}_{k}} \!\!\left( \left\|  \Lambda_{0} \oplus  {\cal M}_{k} \right\| \right) \\
& =\max_{{\cal M}_{k}} \left( \left\|   {\cal M}_{k} \right\| \right)\\
& = \mu_{1} + \cdots + \mu_{k} \end{align*}
and thus the partition defined along the left edge is $\mu$.

Along the right edge of the hive all entries have the form $h_{kk}$. Usinq Equation~\eqref{h max} we see
\begin{align*} h_{kk} & = \max_{\Lambda_{k} \oplus  {\cal M}_{0}} \!\!\left( \left\|  \Lambda_{k} \oplus  {\cal M}_{0} \right\| \right)  \\
& = \max_{\Lambda_{k} } \left( \left\|  \Lambda_{k} \right\| \right)\\
& = \lambda_{1} + \cdots + \lambda_{k}
\end{align*}
and thus the partition defined along the left edge is $\lambda$.
\bigskip

\noindent Along the bottom of the hive all entries have the form $h_{kn}$. Using Equation~\eqref{h min} we see
\begin{align*}  h_{kn} & =  \  | \lambda | - \min_{\substack{ \Lambda_{n-n} \oplus {\scriptstyle{\cal N}_{n-k}}}} \!\!\big(  \left\| \la_{n-n}+ \n_{n-k} \right\| \big) \\
& =  \  | \lambda | - \min_{\substack{  {\scriptstyle{\cal N}_{n-k}}}} \!\!\big(  \left\| \n_{n-k} \right\| \big) \\
&=\  | \mu |+|\nu| - \min_{\substack{  {\scriptstyle{\cal N}_{n-k}}}} \!\!\big(  \left\| \n_{n-k} \right\| \big) \\
& = | \mu| + \nu_{1} + \cdots + \nu_{k}
\end{align*}
and thus the partition defined along the bottom edge is $\nu$.

\bigskip

The proofs for establishing the type of the hive given by Equations~\eqref{h' min} and/or~\eqref{h' max} are proved in the same way. As above, the simplest arguments are found using Equation~\eqref{h' max} for the left and right sides of the hive (giving partitions $\nu$ and $\lambda$, respectively), while it is easiest to see the bottom edge gives the partition $\mu$ by using Equation~\eqref{h' min}.
\end{proof}

\section{Future Questions}
Our chief interest in establishing Theorem~\ref{full theorem} was in connecting our earlier work on Littlewood-Richardson fillings and linear algebra over valuation rings to recent questions in the study of the affine Grassmannian, and also the conjectured formula for hives from the work of Danilov and Koshevoy~\cite{DK}. Many questions and open problems remain.

In our earlier work, we were able to give, from a matrix pair $(M,N)$ over a certain valuation ring, a hive construction of \emph{both} types $(\mu,\nu, \lambda)$ and of $(\nu,\mu,\lambda)$. Further, we were able to show~\cite{lrreal} (by means of a rather delicate argument) that the bijection $c_{\mu \nu}^{\lambda} \leftrightarrow c_{\nu \mu}^{\lambda}$ we constructed matricially exactly matched the combinatorially defined bijection (as described by James and Kerber~\cite{J-K}) known previously. Our Theorem~\ref{full theorem} here seems likely to construct such a bijection between hives of type $(\mu,\nu,\lambda)$ and $(\nu,\mu,\lambda)$ and, indeed, to agree with our previous construction in ~\cite{lrreal}, at least over the rings for which that earlier construction applied. The proposed function would map a hive $\{ h_{st} \}$ of type $(\mu,\nu,\lambda)$ to a hive $\{ h_{st}' \}$ of type $(\nu,\mu,\lambda)$ provided there is a pair of $(\n, \la) \in \Gr \times \Gr$ of the appropriate type for which both hive constructions of Theorem~\ref{full theorem} applied to $(\n, \la)$ yield the hives $\{ h_{st} \}$ and $\{ h_{st}' \}$. That said, among the open problems remaining in this line of inquiry would be to establish:
\begin{enumerate}
\item That the map from $(\n, \la) \in \Gr \times \Gr$ to hives (of either type) is onto.
\item That if two pairs $(\n, \la)$ and $(\n', \la')$ both yield a hive $\{ h_{st} \}$ of type $(\mu,\nu,\lambda)$, that both \emph{also} produce the \emph{same} hive $\{ h_{st}' \}$ of type $(\nu,\mu,\lambda)$ (that is, is the conjectured map from $\Gr \times \Gr$ well-defined?).
\item Given affirmative answers to these questions, does the map reproduce the combinatorial map of James and Kerber~\cite{J-K}?
\end{enumerate}

As stated in the introduction, the precise form of our map (defined by minima or maxima of the orders of invariants of certain submodules) is precisely analogous to a conjectured map relating hives to pairs of Hermitian matrices, first studied by Danilov and Koshevoy~\cite{DK}. The conjecture is apparently still open, but it is our hope that the result here might inspire new avenues for the pursuit of the conjecture. Indeed, our earlier interest in the Hermitian case was led, in part, in an analysis of the effect on the hives associated to Hermitian matrix pairs, under various matrix deformations (rotations of eigenvectors). This analysis suggested an interesting connection between hive deformation and the structure of ${\mathfrak sl}_{n}$ crystals. Even if the Hermitian case remains elusive, our hope would be to take up the hive deformations in the algebraic setting adopted here, where explicit (and discrete) formulas seem more readily available. This analysis might, in turn, shed light on some of the deeper questions considered by Kamnitzer~\cite{kam} relating crystal construction and representation theory under the geometric Satake correspondence.

\end{document}